\documentclass[11pt,letterpaper]{amsart}
 
  \usepackage[color=lightgray, textsize=tiny,textwidth=4cm]{todonotes}
\usepackage{amsmath}
\usepackage[curve]{xypic}
\usepackage{enumitem}
\usepackage{color}
\usepackage{url}
\usepackage{comment}

\makeatletter % makes "@" a normal character so that it can be used in macros
\def\@cite#1#2{{\m@th\upshape\bfseries%
[{#1\if@tempswa{\m@th\upshape\mdseries, #2}\fi}]}}
\makeatother % returns "@" to its normal state

\theoremstyle{plain}
\newtheorem{thm}{Theorem}[section]% subsection
\newtheorem{cor}[thm]{Corollary}
\newtheorem{prop}[thm]{Proposition}
\newtheorem{lem}[thm]{Lemma}
\newtheorem{sublem}[thm]{Sublemma}
\theoremstyle{definition}

\newtheorem{conj}[thm]{Conjecture}
\theoremstyle{remark}
\newtheorem{rem}[thm]{Remark}

\numberwithin{equation}{subsection}
\renewcommand{\bold}[1]{\medskip \noindent {\bf #1 }\nopagebreak}

\newcommand{\nc}{\newcommand}
\newcommand{\rnc}{\renewcommand}
\newcommand{\e}{\varepsilon}

\newcommand{\wrt}[1]{\mathrm{d}{#1}}

\nc\bA{\mathbb{A}}
\nc\bB{\mathbb{B}}
\nc\bC{\mathbb{C}}
\nc\bD{\mathbb{D}}
\nc\bE{\mathbb{E}}
\nc\bF{\mathbb{F}}
\nc\bG{\mathbb{G}}
\nc\bH{\mathbb{H}}
\nc\bI{\mathbb{I}}
\nc{\bJ}{\mathbb{J}} 
\nc\bK{\mathbb{K}}
\nc\bL{\mathbb{L}}
\nc\bM{\mathbb{M}}
\nc\bN{\mathbb{N}}
\nc\bO{\mathbb{O}}
\nc\bP{\mathbb{P}}
\nc\bQ{\mathbb{Q}}
\nc\bR{\mathbb{R}}
\nc\bS{\mathbb{S}}
\nc\bT{\mathbb{T}}
\nc\bU{\mathbb{U}}
\nc\bV{\mathbb{V}}
\nc\bW{\mathbb{W}}
\nc\bY{\mathbb{Y}}
\nc\bX{\mathbb{X}}
\nc\bZ{\mathbb{Z}}
\nc\cA{\mathcal{A}}
\nc\cB{\mathcal{B}}
\nc\cC{\mathcal{C}}
\rnc\cD{\mathcal{D}}
\nc\cE{\mathcal{E}}
\nc\cF{\mathcal{F}}
\nc\cG{\mathcal{G}}
\rnc\cH{\mathcal{H}}
\nc\cI{\mathcal{I}}
\nc{\cJ}{\mathcal{J}} 
\nc\cK{\mathcal{K}}
\rnc\cL{\mathcal{L}}
\nc\cM{\mathcal{M}}
\nc\cN{\mathcal{N}}
\nc\cO{\mathcal{O}}
\nc\cP{\mathcal{P}}
\nc\cQ{\mathcal{Q}}
\rnc\cR{\mathcal{R}}
\nc\cS{\mathcal{S}}
\nc\cT{\mathcal{T}}
\nc\cU{\mathcal{U}}
\nc\cV{\mathcal{V}}
\nc\cW{\mathcal{W}}
\nc\cY{\mathcal{Y}}
\nc\cX{\mathcal{X}}
\nc\cZ{\mathcal{Z}}

\nc{\dmo}{\DeclareMathOperator}
\rnc{\Re}{\operatorname{Re}}
\rnc{\Im}{\operatorname{Im}}
\dmo{\rank}{rank}
\dmo{\End}{End}
\dmo{\Hom}{Hom}
\dmo{\Jac}{Jac}
\dmo{\Id}{Id}
\dmo{\Ann}{Ann}
\dmo{\Area}{Area}
\dmo{\CP}{\bC P^1}
\dmo{\vol}{Vol}
\dmo{\Vol}{Vol}
\dmo{\Isom}{Isom}
\dmo{\logp}{\overline{\log}}

\title{Towards optimal spectral gaps in large genus}
\author[Lipnowski]{Michael~Lipnowski}
\author[Wright]{Alex~Wright}
\begin{document}
\maketitle
% removes page number from first page
\thispagestyle{empty}

%\vspace{-0.9cm}

% TO DO LIST: 

% -- Mike check over Alex's changes is Sections 1-4 (all changes in red and annotated) 

% -- Mike make any changes required to Section 5 (ideally read whole section) and then Alex look over

% -- In appendix A: 
% -- -- maybe good idea to keep dependence on n until the end
% -- -- clean up a bit
% -- -- write proof for B bound (maybe we can comment it out after) 
% -- -- maybe need some more consolodiation statements, reorder, ... 
% -- -- update statement of main theorem at beginning (remove duplication)
% -- -- make sure new statement is good enough for sinh bound, see if changes to sinh bound are required
% -- -- (and double check that the only place we use the theorem in the sinh bound) 

% *** In appendix B: Mike read proof of sinh bound to double check its all good; Alex check rest of appendix B

% *** In appendix C: Mike look over and make any required changes 

% \begin{abstract}
%  We show that the Weil-Petersson probability that a random surface has first eigenvalue of the Laplacian less than $3/16-\epsilon$ goes to zero as the genus goes to infinity. 
% \end{abstract}

\vspace{-0.4cm}
%%%%%%%%%%%%%%%%%%% 
% TABLE OF CONTENTS
%%%%%%%%%%%%%%%%%%%
% allows subsections (depth 1) to be displayed in table of contents
\setcounter{tocdepth}{1} 
\tableofcontents
%\vfill
%\newpage
%\vspace{1cm}
%To do: 
%\begin{itemize}
%%\item 
%\end{itemize}

%Questions: 
%\begin{itemize}
%\item
%\end{itemize}

%\newpage

\vspace{-0.6cm}
\section{Introduction} 

\bold{Main result.} The first non-zero Laplace eigenvalue  $\lambda_1$ of a hyperbolic surface controls the speed of mixing of geodesic flow, the error term in the Geometric Prime Number Theorem, and measures the extent to which the surface is an expander. 

In high genus, the best one can hope for is that $\lambda_1$ is close to $\frac14$. Indeed, if $\Lambda=\limsup_{g\to \infty} \Lambda_g$, where $\Lambda_g$ denotes the maximum value of $\lambda_1$ over  the moduli space  $\cM_g$  of genus $g$ closed hyperbolic surfaces,   then we have  $\Lambda \leq \frac14$  \cite{Huber, Cheng}. 

It is natural to conjecture that $\lambda_1$ is typically close to this optimal value of $\frac14$ in large genus, especially since the corresponding statement for regular graphs is true \cite{Friedman, Bor}. Despite extensive study of Selberg's related eigenvalue $\frac14$ conjecture, it was only proven recently, by Hide and Magee \cite{MageeHide} after this paper was written, that $\Lambda=\frac14$.  However, the behavior of $\lambda_1$ for Weil-Petersson typical random surfacs of large genus remains an open problem.

 % (or even $\Lambda > \frac{1}{4} - \left( \frac{7}{64} \right)^2$).

In this paper, we study this conjecture by averaging the Selberg trace formula over $\cM_g$ and using ideas originating in Mirzakhani's thesis.  We establish the following:

\begin{thm}\label{T:main}
For all $\e>0$, the Weil-Petersson probability that a surface in $\cM_g$ has $\lambda_1<\frac3{16}-\e$ goes to zero as $g\to\infty$. 
\end{thm}

The same result was obtained independently by Wu and Xue  \cite{WuXue}. Previously Mirzakhani showed the same result with $\frac3{16}$ replaced with $\frac14\left(\frac{\log(2)}{2\pi+\log(2)}\right)^2 \approx 0.002$  \cite{Mirzakhani:Growth}. Related results for random covers of a fixed surface, again with $\frac3{16}$ appearing, were obtained for closed surfaces by Magee,  Naud, and Puder in \cite{MNP} and for convex cocompact surfaces by Magee and Naud \cite{MN}.

%For results on random covers of convex cocompact hyperbolic surfaces, see \cite{MN}. 

\bold{Idea of the proof.} Our work is inspired by and builds on recent work of Mirzakhani and Petri \cite{MirzakhaniPetri:Lengths}. They fix a constant length $L$, and consider geodesics of length at most $L$. As the genus goes to infinity, they show in particular that, averaged over $\cM_g$,
\begin{enumerate}
\item most geodesics of length at most $L$ are simple and non-separating, and 
\item the number of simple non-separating geodesics of length at most $L$ can be estimated using Mirzakhani's integration formula.
\end{enumerate}

This paper extends these observations to lengths scales $L$ that grow slowly with genus. As the error term in the Geometric Prime Number Theorem suggests,  bounds on the number of geodesics translate into bounds on $\lambda_1$.   

The starting point for the first observation above are computations that show, on average, there aren't too many subsurfaces which a non-simple geodesic of length at most $L$ can fill. Given this, one must show that most such subsurfaces don't have too many closed geodesics. This is more difficult when $L$ grows with genus, and requires that we establish new bounds in Section \ref{S:Bounds}. 

Even though our analysis shows that the contribution of the non-simple geodesics is a lower order term at the length scales we consider, it may be nescessary to better understand this term to move beyond the $\frac3{16}$ barrier. 

\bold{Broader significance.} Our results are broadly applicable to any problem that relates to counting geodesics or that might be studied by averaging the Selberg Trace Formula over $\cM_g$, and provide tools towards improved error terms in limit multiplicity laws \cite{Monk, Seven}, calculations with error terms for the average number of geodesics with lengths in intervals with sizes growing or shrinking with the genus, and first steps towards understanding eigenvalue spacing \cite{Sarnak}. 

Additionally, we believe our results concerning the nature of geodesics of different lengths scales are just the tip of the iceberg. We suggest the following as an accessible starting point for further investigation.  % \textcolor{red}{Add a bridge sentence before Conjecture 1.2?} 

\begin{conj}
As $g$ goes to infinity, on most surfaces in $\cM_g$ most geodesics  of length significantly less than $\sqrt{g}$ are simple and non-separating, and most geodesics of length significantly greater than $\sqrt{g}$ are not simple.\footnote{Added in proof: this conjecture and its analogue for graphs have been confirmed in \cite{BPsurfaces, BPgraphs}.} 
\end{conj}

Because error terms are central in our analysis, a proof of this conjecture will not necessarily yield any improvements to Theorem \ref{T:main}. However it would improve our  understanding of high genus surfaces.

\bold{Structure of the proof.} The proof of Theorem \ref{T:main} is  divided into a geometric bound on geodesics  and an argument using the Selberg Trace Formula. We state these results now. 

Given a compactly supported real function $F$, define $F_{\mathrm{all}}:\cM_g \to \bR$ by setting $F_{\mathrm{all}}(X)$ to be the sum of $F$ over the lengths of all oriented closed primitive geodesics on $X$. When not otherwise specified, we refer to the Weil-Petersson measure on $\cM_g$. 

\begin{thm}\label{T:GeometricMain} 
For any constants $D>0$ and $1>\kappa>0$, 
there are subsets $\cM_g'$ of $\cM_g$ such that 
\begin{enumerate}
\item ${\Vol(\cM_g')}/\Vol(\cM_g) \to 1,$
\item every surface in $\cM'_g$ has systole at least $1/\log(g)$, and
\item for any non-negative function $F$ with support in $[0, D\log(g)]$, the average of $F_{\mathrm{all}}$ over $\cM_g'$ is at most 
$$I_F + O(g^{-1+\kappa}I_{\widetilde{F}}), $$
where 
$$I_F = \int_0^\infty F(\ell) \ell  \frac{\sinh(\ell/2)^2}{(\ell/2)^2} d\ell$$
 and $\widetilde{F}(\ell)=\sup_{[\ell-1,\ell+1]}F$.  %\|F|_{\ell-1,\ell+1}\|_\infty$
% $$(1+O(g^{-1+\kappa})) I_F,$$
% where 
% $$I_F = \int_0^\infty F(\ell) \ell  \frac{\sinh(\ell/2)^2}{(\ell/2)^2} d\ell.$$
\end{enumerate}
\end{thm}

\begin{thm}\label{T:SpectralMain}
Fix $D > 4$ and $1 > \kappa > 0.$  Let $\mu$ be a Borel probability measure on $\cM_g$ such that
\begin{enumerate}
\item $\mu$ is supported on the $e^{- g^{o(1)}}$-thick part of $\cM_g$, and
\item for any non-negative function $F$ with support in $[0, D\log(g)]$, the $\mu$-average of $F_{\mathrm{all}}$ is at most 
$$I_F + O(g^{-1+\kappa}I_{\widetilde{F}}). $$
%$$(1+O(g^{-1+\kappa})) I_F.$$
\end{enumerate}
Then the $\mu$-probability that $\lambda_1(X) \leq \frac{1}{4} - b^2$ is at most $$O \left(  g^{ 1 - 4b \left(1 - \frac{\kappa}{2} \right) + o(1)} \right),$$
where the implicit constant in the big O notation depends only on $D$, $\kappa$, and the implicit constant in $O(g^{-1+\kappa})$.
\end{thm}

\begin{proof}[Proof of Theorem \ref{T:main} given Theorems \ref{T:GeometricMain} and \ref{T:SpectralMain}]
Let $\mu_g$ be the restriction of the Weil-Petersson measure to $\cM_g'$, normalized to be a probability measure.  Theorem \ref{T:SpectralMain} proves that the $\mu_g$ probability that $\lambda_1<\frac3{16}-\e$ goes to zero as $g\to \infty$. Since the complement of  $\cM_g'$ has probability measure going to zero as $g\to \infty$, this gives the result. 
\end{proof}

\bold{Additional context.} Mirzakhani pioneered the study of Weil-Petersson random surfaces \cite{Mirzakhani:Growth}, and devoted her ICM address to this topic \cite{Mirzakhani:ICM}. Building on her previous study of Weil-Petersson volume polynomials, she proved in particular that the probability that the Cheeger constant is smaller than $0.099$
%$\frac{\log(2)}{2\pi+\log(2)}-\e$ 
goes to zero as the genus goes to infinity.  

More recent works motivated by the problem of understanding $\lambda_1$ of typical high genus surfaces include results on Weil-Petersson volume polynomials \cite{AM} and the geometry of typical surfaces \cite{MT}.

For some additional open problems related to random surfaces, see \cite[Section 10.4]{Tour}. In analogy with regular graphs \cite{Alon}, we expect that Riemann surfaces of high genus have Cheeger constant bounded away from 1, and that Theorem \ref{T:main} cannot be obtained using the Cheeger inequality.\footnote{ Added in proof: This expectation was proven correct in \cite{NewCheeger}.}  See \cite{NWX, WuXueParlier} for a recent study of separating geodesics. 

Additional results on small eigenvalues can be found in \cite{WuXue2, Dub}.

 \bold{Key tools.}  In Sections \ref{S:thin} and \ref{S:GeometricMain}, we  use  the formula Mirzakhani gave in her thesis for integrating certain functions over $\cM_g$ \cite{Mirzakhani:Invent}.  We briefly recall this formula now, but see also  the second author's expository introduction to Mirzakhani's work for an introduction offering more explanation and intuition   \cite[Section 4]{Tour}. 

 Let $\cM_{g,n}$ denote the moduli space of genus $g$ hyperbolic surfaces with $n$ cusps.  
Let $V_{g,n}$ denote the volume of $\cM_{g,n}$, and $V_{g,n}(L_1, \ldots, L_n)$ denote the volume of the moduli space of genus $g$ hyperbolic surfaces with boundary geodesics of lengths $L_1, \ldots, L_n$. Given a compactly supported real function $F$, define $F_{\mathrm{sns}}:\cM_g \to \bR$ by letting $F_{\mathrm{sns}}(X)$ be the sum of $F$ over the lengths of all oriented simple non-separating geodesics on $X$. A special case of Mirzakhani's integration formula is 
$$\int_{\cM_g} F_{\mathrm{sns}} = \int_0^\infty \ell F(\ell) V_{g-1,2}(\ell, \ell) d\ell.$$

  More generally, if $\gamma=(\gamma_1, \ldots, \gamma_k)$ is a tuple of disjoint simple curves, and $F$ is  a  function of $k$ real variables,  we define $F_\gamma: \cM_{g}\to \bR$ by  
$$F_\gamma(X) = \sum_{\alpha} F(\ell_{\alpha_1}(X), \ldots, \ell_{\alpha_k}(X)),$$
where the sum is taken over the mapping class group orbit of $\gamma$.
%
% Suppose that cutting the geodesic representative of $\gamma$ decomposes $X$ into $s$ connected components $X_1, \ldots, X_s$, and that 
% \begin{itemize}
% \item $X_j$ has genus $g_j$,
% \item $X_j$ has $n_j$ boundary components, and 
% \item the lengths of the boundary components of $X_j$ are given by $\Lambda_j\in \bR_+^{n_j}$.
% \end{itemize}
%$X_j$ has genus $g_j$ and has $n_j$ boundary components whose lengths are given by $\Lambda_j\in \bR_+^{n_j}$. 
%If we set $\ell_i = \ell_{\gamma_i}(X)$, then all the entries of each $\Lambda_j$ are from $\{L_1, \ldots, L_n\}$ (if they correspond to the original boundary of $X$) or $\{\ell_1, \ldots, \ell_k\}$ (if they correspond to the new boundary created by cutting $\gamma$).
%
Mirzakhani's Integration Formula gives that
$$\int_{\cM_{g}} F_\gamma 
=\iota_\gamma \int_{\ell=(\ell_1, \ldots, \ell_k)\in \bR^k_+} \ell_1 \cdots \ell_k F(\ell_1, \ldots , \ell_k) \prod_{j=1}^s V_{g_j, n_j}(\Lambda_j) \wrt\ell,
$$
where $\iota_\gamma\in\bQ_+$ is an explicit constant, $s$ is the number of connected components of the complement of $\gamma$, and $V_{g_j, n_j}(\Lambda_j)$ are the volumes of the moduli spaces which naturally contain those components.

 We also use the  asymptotics of $V_{g,n}$  established  in \cite[Theorem 1.8]{MirzakhaniZograf:LargeGenus}.  That result gives a constant $C>0$ such that for any sequence $n(g)$ with $$\lim_{g\to\infty} \frac{n(g)^2}{g}=0,$$ the estimate
\begin{equation}
\label{E:asym}
V_{g,n(g)} = \frac{C}{\sqrt{g}} (2g-3+n(g))! (4 \pi^2)^{2g-3+n(g)}\left( 1+ O\left( \frac{1+n(g)^2}{g} \right)  \right)
\end{equation}
holds as $g\to \infty$. We also use that, as a consequence of the asymptotics and \cite[inequality (3.8)]{Mirzakhani:Growth}, there is a constant $C_0>0$ such that, for \emph{all} $g$ and $n$ we have the upper bound  
\begin{equation}
\label{E:upper}
V_{g,n} \leq \frac{C_0}{\max(1,\sqrt{g})} (2g-3+n)! (4 \pi^2)^{2g-3+n}.    
\end{equation}

\bold{Organization.}  
In Section \ref{S:Bounds} we give bounds on the number of closed geodesics on surfaces with boundary, showing that often there are vastly fewer geodesics than one  expects on closed surfaces. (In particular, the critical exponent is often close to $0$.)

These bounds however require a lower bound on the systole. For this reason, we require estimates for a version of Mirzakhani's Integration Formula over the thin part of $\cM_g$, which we obtain in Section \ref{S:thin} via an inclusion-exclusion argument. This section also gives more precise estimates for the volume of the thin part than were previously known. 

Sections \ref{S:GeometricMain} and \ref{S:SpectralMain} prove Theorems \ref{T:GeometricMain} and \ref{T:SpectralMain} respectively. 

Our arguments in Sections \ref{S:thin} and \ref{S:GeometricMain} crucially rely on estimates of Mirzakhani \cite{Mirzakhani:Growth}, Mirzakhani-Zograf \cite{MirzakhaniZograf:LargeGenus}, and Mirzakhani-Petri \cite{MirzakhaniPetri:Lengths}. For the convenience of the reader, we revisit the proofs of these estimates in  Appendices \ref{A:VolPoly} and \ref{A:Vols} to verify some results on  uniformity  that were not explicitly included in the original statements. Similarly, in Appendix \ref{localweyllawappendix} we review a standard local Weyl law argument used in Section \ref{S:SpectralMain}. 

\bold{Acknowledgements.}  We thank Farrell Brumley, Andrew Granville, Rafe Mazzeo, Peter Sarnak, and Scott Wolpert for helpful conversations. We also especially thank Paul Apisa  and the referees  for detailed and helpful comments. 

During the preparation of this paper, the first author was partially supported by a NSERC Discovery Grant, and the second author was partially supported by a Clay Research Fellowship,  NSF Grant DMS 1856155, and a Sloan Research Fellowship. 

\section{Surfaces with few geodesics}\label{S:Bounds}

Throughout this paper, all hyperbolic surfaces and subsurfaces are assumed to be compact, and are either closed or have geodesic boundary.  The purpose of this section is to show the following theorem. 

\begin{thm}\label{T:GraphBound}
For any $A>0$, there exists  a $C>0$ such that if $X$ is a hyperbolic surface of area $A$  and $L_0>1$ and $\frac12>\e>0$ are such that
\begin{enumerate}
\item $X$ does not have any pants or one-holed tori of total boundary length less than $L_0$, and
\item  $X$ has systole at least $\e$,
\end{enumerate}
 then for all $\ell>0$, the number of closed geodesics on $X$ of length at most $\ell$ is at most 
 $$\left(\frac{C L_0 \log(1/\e)}{\e}\right)^{C \ell/L_0+3}.$$
\end{thm}%{We can get rid of the $\log(1/\e)$ if we want to.} 

%Our proof is in principal effective, but we do not make any attempt to achieve the best possible constants. %We require only quite soft bounds. 

\begin{rem}
When getting upper bounds of the form $O(e^{\delta \ell})$ on the number of closed geodesics on a general hyperbolic surface, it isn't possible to do  better than $\delta=1$. But, for fixed $A$ and $\e$, Theorem \ref{T:GraphBound} gives such a bound with $\delta$ a multiple of $\log(L_0)/L_0$. %In particular, this gives an upper bound on the critical exponent of $X$, but it is important for us that we have an upper bound valid at all length scales and not just an asymptotic. 
\end{rem}

% It is easy to see that $L_0$ can only be large if $X$ has long boundary, assuming a bound on $\e$ and $A$.  We think of Theorem \ref{T:GraphBound} as saying  that many surfaces with long boundary have very few geodesics. 
It is well known that  for all $g$ there exists an $L_0$ such that if $X$ is a closed surface of genus $g$ then $X$ has a pants of total boundary length at most $L_0$ (see for example \cite[Chapter 5]{Buser}). Thus, since Theorem \ref{T:GraphBound} only gives useful information when $L_0$ is large enough to outweigh the influence of the unknown constant $C$, Theorem \ref{T:GraphBound} does not give any information at all on closed surfaces. It is also easy to see that for all $g$ and $n$ and $B>0$ there exists an $L_0$ such that if $X$ is a genus $g$ surface with $n$ boundary components whose total boundary length is at most $B$, then $X$ has a  pants of total boundary length at most $L_0$. Again since $C$ is unknown, this means Theorem \ref{T:GraphBound} does not give any information for any fixed values of $g,n$ and $B$. To the contrary, the reader should have in mind the situation when $g$ and $n$ are fixed but $B$ goes to infinity, and   we think of Theorem \ref{T:GraphBound} as saying  that many surfaces with long boundary have very few geodesics.

Surfaces satisfying the first condition in Theorem \ref{T:GraphBound} are studied in \cite{MT}, where they are called $L_0/2$-tangle-free.

\subsection{The thin part}
If $X$ is a hyperbolic surface with boundary, then we may form the double $X_d$ of $X$ by gluing together two copies of $X$ along their boundary to obtain a closed surface. The double is equipped with an involution whose quotient is $X$. For any $\delta$, the $\delta$-thin part of $X_d$ is defined to be the subset where the injectivity radius is less than $\delta$, and the $\delta$-thin part of $X$ is the image of this set in $X$.

For $\delta$ small enough, the $\delta$-thin part of $X_d$ is a disjoint union of collars around short geodesics.  (See, for example, \cite[Section 2.2]{Tour} for a brief introduction to this fact with an illustration, and \cite[Chapter 4]{Buser} for more details.)   Each such collar is either fixed by the involution or exchanged with another collar. In the fixed case,  each of the two boundary circles is mapped to itself (since otherwise cone points would appear in the quotient),  and we call the quotient of the collar a thin rectangle.  

 We thus see that the $\delta$-thin part  of $X$ is a union of collars and thin rectangles. Thin rectangles correspond to regions of $X$ where two segments of the boundary of $X$ are very close to each other. We define the width of a rectangle to be the minimal distance between the components of its boundary in the interior of $X$, so very thin rectangles have very large width. 

From now on, we fix $\delta$ and refer to the $\delta$-thin part simply as the thin part. Its complement is called the thick part. Recall  the following standard fact. 

\begin{lem}
For all $A>0$ there exists a constant $C$ such that if $X$ is a hyperbolic surface of area at most $A$ with boundary, then there is a set of 
at most $C$ points on $X$ such that every point in the thick part is within distance $1$ of one of these points. 
\end{lem}

\begin{proof}[Proof sketch]
Any maximal 1-separated subset of the thick part will work. Since the $\min(\delta, \frac12)$-balls around these points are embedded (in the double) and disjoint, an area estimate gives an upper bound for the number of points in the subset.
\end{proof}

From this we immediately get the following, keeping in mind that a  geodesic  of length $\e$ has a collar of  diameter comparable to  $\log(1/\e)$  and that the number of short geodesics is (linearly) bounded by $A$. 

\begin{cor}\label{C:Net}
For all $A>0$ there exists a constant $C$ such that that if $X$ is a hyperbolic surface of area at most $A$ with boundary, $L_0>1$ is arbitrary, and the systole of $X$ is at least $\e>0$, then there is a set of 
at most $$C+C \log(1/\e)/L_0$$ points on $X$ such that every point not in a thin rectangle of width at least $L_0$ is within distance $L_0/{48}$ of one of these points. 
\end{cor}

For the remainder of this section, given a surface $X$ as in Theorem \ref{T:GraphBound}, and a choice of $L_0$, we fix a set $\mathrm{Net}(X)$ of points as in Corollary \ref{C:Net}.

Each thin rectangle has two boundary components in the interior of $X$, which we think of as the two ends of the rectangle. For each end of each rectangle of width at least $L_0$,  add a point on this end  to $\mathrm{Net}(X)$. These points will be distinguished in that we will remember that that each such point is associated to the end it lies on. Since the number of thin rectangles is bounded linearly in terms of $A,$ we have added at most a constant (linear in $A$) number of points to $\mathrm{Net}(X)$.  (The number of thin rectangles is at most equal to the maximal number of short geodesics in the double.)  

\subsection{Good segments}

Define a \emph{good segment} to be a geodesic segment joining two points in $\mathrm{Net}(X)$ that either
\begin{itemize}
\item has length at most $L_0/12$, or 
\item is contained in a thin rectangle of width at least $L_0$ and starts and ends at the chosen points on each end. 
\end{itemize}
The zero length geodesic joining a point in $\mathrm{Net}(X)$ to itself will be considered a good segment. 

The purpose of this subsection is to show the following. 

\begin{prop}\label{P:SegmentCount} 
Let X be as in Theorem \ref{T:GraphBound}. For any two points $p_1, p_2\in \mathrm{Net}(X)$, there are at most 
$$3+\frac{L_0}{6\e}$$
good segments joining $p_1$ and $p_2$. 
\end{prop}

\begin{lem}\label{L:AnnulusCount}
Let $\gamma \in \Isom^+(\bH)$ have translation length $T>0$. Then, for any two points in $\bH/\langle \gamma \rangle$, there are at most 
$$2+\frac{2\ell}{T}$$
geodesic segments of length at most $\ell$ joining these two points. 
\end{lem}

\begin{proof}
The projection from $\bH/\langle \gamma \rangle$ onto its unique closed geodesic is distance non-increasing, so it suffices to assume the two points lie on this geodesic. In each of the two possible directions along this geodesic, the segment must first go from one point to the other, and then can make at most $\ell/T$ complete loops around the geodesic. 
\end{proof} 

The following elementary observation applies both when $X$ is closed and when it has geodesic boundary and cusps.

\begin{lem}\label{L:BallGood}
Suppose that a hyperbolic surface $X$ doesn't have any pants or one-holed tori with total boundary length less than $L_0$. Then the ball $B$ of radius  $R=L_0/12$ centered at any point  $p\in X$ is isometric to either a subset of $\bH$ or to a subset of $\bH/\langle \gamma \rangle$ for some $\gamma \in \Isom^+(\bH)$.
\end{lem}

A version of this lemma appears in \cite[Proposition B]{MT}.

\begin{proof}
It suffices to consider the case when $X$ doesn't have boundary. 

In order to find a contradiction, assume that $B$ is not homeomorphic to a ball or an annulus. 

\begin{sublem}
$B$ contains two simple loops $\alpha_1$ and $\alpha_2$, each of length at most $2R$, that intersect at most once. 
\end{sublem}

\begin{proof}
Slowly grow an open ball centered at $p$, starting with small radius and then increasing the radius to $R=L_0/12$. Let $R_1<R$ be the maximum radius where this ball is embedded, so the closure of the ball of radius $R_1$ contains a point $q_1$ that appears with multiplicity at least two in the boundary circle of the ball. Define $\alpha_1$ to be the simple loop that travels from $q_1$ to $p$ and then back out to the other appearance of $q_1$ on the boundary of the ball. See Figure \ref{F:alpha12} (left). 

\begin{figure}[ht!]
\includegraphics[width=\linewidth]{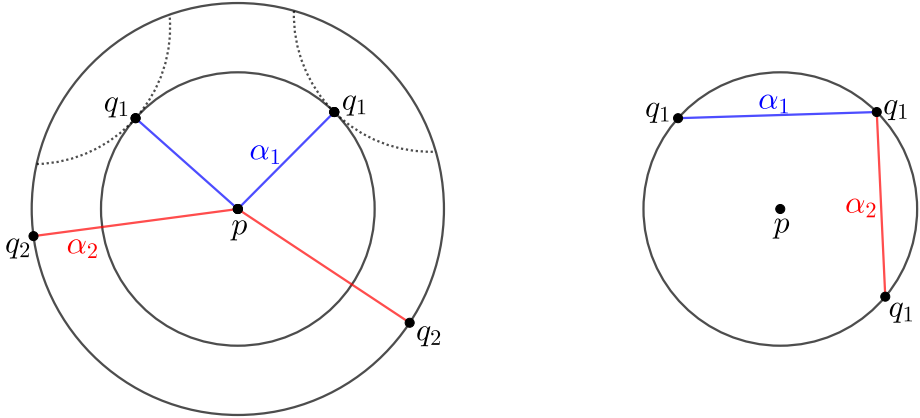}
\caption{Finding $\alpha_1$ and $\alpha_2$.}
\label{F:alpha12}
\end{figure}

If $q_1$ is not unique, then we can similarly define $\alpha_2$. If $q_1$ appears with multiplicity greater than 2, we define $\alpha_2$ as in Figure \ref{F:alpha12} (right). 

Otherwise, the ball of radius slightly larger than $R_1$ is topologically an annulus. Let $R_2<R$ be the maximum radius where this remains true.  So there is a point $q_2$ which appears at least twice  on the boundary of the ball of radius $R_2$ and isolated among such points. We now define $\alpha_2$ analogously to $\alpha_1$. 
\end{proof}

If $\alpha_1$ and $\alpha_2$ intersect once, then  either  the boundary of the neighborhood of their union  is connected and  can be tightened to a geodesic of length at most $$2\ell(\alpha_1)+2\ell(\alpha_2)\leq 8R,$$ and this geodesic bounds a torus,  or this  boundary  has three connected components which again can be tightened to geodesics  of total length at most $8R$, and these three geodesics  bound a pants.  

If they don't intersect, we can form a simple curve $\alpha_3$ by going around $\alpha_1$, then taking a minimal length path $\beta$ to $\alpha_2$, going around $\alpha_2$ once, and then going back to $\alpha_1$ along $\beta$. This $\alpha_3$ has length at most 
$$\ell(\alpha_1) + \ell(\alpha_2) + 2\ell(\beta) \leq 8R.$$
The geodesic representatives of  $\alpha_1, \alpha_2, \alpha_3$ bound a pants with total boundary at most $12R$. 
\end{proof}

\begin{proof}[Proof of Proposition \ref{P:SegmentCount}.] 
Call the two points $p_1$ and $p_2$. There is at most 1 good segment joining $p_1$ and $p_2$ of length greater than $L_0/12$; this segment can only exist if $p_1$ and $p_2$ are the points associated to opposite ends of a thin rectangle. 

So we now count the number of geodesics joining $p_1$ and $p_2$ of length at most $L_0/12$. Any such geodesic is of course contained in the ball of radius $L_0/12$ about $p_1$. 

By Lemma \ref{L:BallGood}, this ball is isometric to a subset of  $\bH$ or to a subset of $\bH/\langle \gamma \rangle.$ In the first case, there is only one geodesic from $p_1$ to $p_2$, and in the second case Lemma \ref{L:AnnulusCount} (with $T\geq\e$ and $\ell=L_1/12$) gives the necessary bound. 
\end{proof}

\subsection{Loops of good segments}
We now relate closed geodesics to good segments.

\begin{lem}\label{L:Homotope} 
Every closed geodesic $\gamma$ of length at most $\ell$ 
is homotopic to a loop of at most 
$$2+\frac{24\ell}{L_0}$$
good segments. 
\end{lem}

\begin{proof}
There is some point $p_1 \in \mathrm{Net}(X)$ of distance at most $L_0/48$ from a point $p_1'$ of $\gamma$. Pick an orientation along $\gamma$. Having defined $p_i\in \mathrm{Net}(X)$ and $p_i'\in \gamma$, define $p_{i+1}'$ and $p_{i+1}$ as follows:
\begin{enumerate}
\item If  $p_1'$ is within distance $L_0/24$ of $p_i'$ in the forward direction along $\gamma$ starting at $p_i'$,  set $p_{i+1}=p_{1}$ and $p_{i+1}'=p_{1}'$, and conclude the construction. 
\item If the point distance $L_0/24$ along $\gamma$ in the forward direction starting at $p_i'$  is not in a thin rectangle of width at least $L_0$, then define $p_{i+1}'$ to be this point. Pick $p_{i+1}$ to be any point of $\mathrm{Net}(X)$ within distance $L_0/48$ of $p_{i+1}'$. 
\item In the remaining case, if $p_i$ is not the point of $\mathrm{Net}(X)$ associated to the entry end of this  thin  rectangle, set $p_{i+1}\in \mathrm{Net}(X)$ to be this  point on the entry end; and otherwise let $p_{i+1}\in \mathrm{Net}(X)$ be the point associated to the exit end of the rectangle. In either case, let $p_{i+1}'$ be a point in $\gamma$ of distance at most 1 away from $p_{i+1}$.  
\end{enumerate}

For each $i$, fix a path from $p_i$ to $p_i'$ of minimal length; this length must be at most $L_0/48$ in all cases. Define $\gamma_i$ be the geodesic representative of the path which goes from $p_i$ to $p_i'$, then goes along $\gamma$ to  $p_{i+1}'$, then goes to $p_{i+1}$. So, by definition $\gamma$ is homotopic to the concatenation of the $\gamma_i$. 

\begin{sublem}
Each $\gamma_i$ is a good segment. 
\end{sublem}

\begin{proof}
Since $1/48+1/24+1/48 = 1/12$, $\gamma_i$ can only have length greater than $L_0/12$ in the final case above, when it crosses a thin rectangle, and in this case $\gamma_i$ is good by definition. 
\end{proof}

It now suffices to bound the number of segments $\gamma_i$, or equivalent the number of points $p_i' \in \gamma$. Suppose the number of such points is $n$. The fact that the distance from $p_{n}'$ to $p_1'$ may be arbitrarily small slightly complicates the bound, since it means the last segment of $\gamma$ is unusual. 

If $i\leq n-2$, then either the distance along $\gamma$ from $p_i'$ to $p_{i+1}'$  is exactly $L_0/24$, or the distance from $p_i'$ to $p_{i+2}'$ is at at least $L_0.$ In this way we see that the average length of either the first $n-2$ or the first $n-1$ of these distances is at least $L_0/24$, and hence we get $(n-2)  L_0/24 \leq \ell$. 
\end{proof}

\begin{proof}[Proof of Theorem \ref{T:GraphBound}]
By Lemma \ref{L:Homotope}, it suffices to bound the number of loops consisting of at most $n=\lfloor 2+\frac{24\ell}{L_0}\rfloor$ good segments. We will actually bound the number of loops with a choice of basepoint in $\mathrm{Net}(X)$, which is larger. Since we allow zero length good segments, we can assume there are exactly $n$ good segments in the loop. 

The number of paths in $\mathrm{Net}(X)$ that can be traced out by such a loop is bounded by $|\mathrm{Net}(X)|^n$. Hence, Proposition \ref{P:SegmentCount} gives that the total number of such loops is at most 
$$|\mathrm{Net}(X)|^n \left(3+\frac{L_0}{6\e}\right)^n.$$
It now suffices to note that, since $\e<\frac12$  and $L_0>1$, 
$$|\mathrm{Net}(X)| \left(3+\frac{L_0}{6\e}\right)    = \left(C+\frac{C \log(1/\e)}{L_0}\right) \left(3+\frac{L_0}{6\e}\right)$$
can be bounded by 
$$\frac{C' L_0 \log(1/\e)}{\e}$$
for some different constant $C'$. 
\end{proof}

\section{Integrating over the thin part}\label{S:thin}

% If someone complains about constants in the integration formula, we should 
% make everything ordered everywhere, and then divide by factorials later

Let $\cM_g^{<\e}$ denote the subset of $\cM_g$ where the surface has a closed geodesic of length less than $\e$. 
The purpose of this section is to prove the following result. 

\begin{thm}\label{T:ThinExpectedValue}
There is a constant $\e_0>0$ such that
for $\e<\e_0$ and $F$ a non-negative function with support in $[0, g^{o(1)}]$, the average of $F_{\mathrm{sns}}$  over  $\cM_g^{<\e}$  is at least  
$$(1-g^{-1+o(1)})  I_F.$$
\end{thm}

More formally, this means that for any function $s(g) \in o(1)$, there exists a function $p(g) \in o(1)$ such that if the support is in $[0,g^{s(g)}]$ then the average is at least $(1-g^{-1+p(g)})  I_F$. The function $p(g)$ does not depend on $\e$. 

\begin{cor}\label{C:ThickAverage}
With the same assumptions, the average of $F_{\mathrm{sns}}$ over $\cM_g^{>\e}$ is at most 
$${(1+g^{-1+o(1)})} I_F $$
\end{cor}

\begin{proof}
Let $\delta$ denote the measure of $\cM_g^{<\e}$ divided by the measure of $\cM_g$, so the desired average is 
$$\frac1{1-\delta} \left(\frac1{V_g} \int_{\cM_g} F_{\mathrm{sns}} - \frac{\delta V_g} {V_g} \frac1{\delta V_g} \int_{\cM_g^{<\e}} F_{\mathrm{sns}} \right).$$
Mirzakhani's integration formula gives
$$\frac1{V_g} \int_{\cM_g} F_{\mathrm{sns}} = \frac{1}{V_g} \int_0^\infty F(\ell) \ell V_{g-1, 2}(\ell, \ell) d\ell.$$
The volume asymptotics with error term in \cite[Theorem 1.8]{MirzakhaniZograf:LargeGenus} imply that 
\begin{equation}\label{E:aratio}
\frac{V_{g-k,2k}}{V_g} = 1 + O\left( \frac{k^2}g\right).
\end{equation}
This statement with $k=1$  together with  the sinh upper bound (Lemma \ref{L:sinh}) gives that
$$\frac1{V_g} \int_{\cM_g} F_{\mathrm{sns}} \leq I_F(1+g^{-1+o(1)}).$$ 
Theorem \ref{T:ThinExpectedValue} gives that 
$$\frac1{\delta V_g} \int_{\cM_g^{<\e}} F_{\mathrm{sns}} \geq (1-g^{-1+o(1)})  I_F.$$
Thus the desired average is at most 
\begin{eqnarray*}
&&\frac{I_F}{1-\delta} \left(1 + g^{-1+o(1)} - \delta(1-g^{-1+o(1)}) \right) 
\\&=& \frac{I_F}{1-\delta} \left(1-\delta + g^{-1+o(1)}+ \delta g^{-1+o(1)} \right),
\end{eqnarray*}
giving the result. (Since $\e$ is small we can assume $\delta < 1/2$.)
\end{proof}

The average in Theorem \ref{T:ThinExpectedValue} is the integral over $\cM_g^{<\e}$  divided by the measure of $\cM_g^{<\e}$, and we estimate the numerator and denominator separately. We will always assume $\e$ is small enough so that geodesics of length at most $\e$ are simple and pairwise disjoint.

\begin{prop}\label{P:ThinMeasure}
There is a constant $\e_0>0$ such that
for $\e<\e_0$, the volume of  $\cM_g^{<\e}$ is
$$\left(1+O \left(g^{-1+o(1)}\right) \right) \cdot (1-\exp(-I_\e)) V_g,$$
where 
$$I_\e = \frac12 \int_0^\e \delta \frac{\sinh(\delta/2)^2}{(\delta/2)^2} d\delta.$$
\end{prop}

The implicit $o(1)$ function does not depend on $\e$. When $\e$ is small, $I_\e$ is about $\e^2/4$. 
For fixed $\e$, this agrees with the asymptotic for the volume of $\cM_g^{<\e}$  obtained in \cite{MirzakhaniPetri:Lengths}.

Both $g^{-1+o(1)}$ and $O(g^{-1+o(1)})$  denote terms within a sub-polynomial factor of $g^{-1}$, but the  latter   does not specify the sign of the term.

\begin{proof}
Let $A_k$ be the integral over $\cM_g$ of the function which counts the number of sets $S$ of $k$ disjoint unoriented geodesics, each of length at most $\e$, on a surface $X\in \cM_g$.

\begin{lem}\label{L:IE}
If $n(g)=3g-3$, then the volume of $\cM_g^{<\e}$ is exactly
$$\sum_{k=1}^{n(g)} (-1)^{k+1} A_k.$$
The same sum with $n(g)$ any odd integer gives an upper bound, and with $n(g)$ even it gives a lower bound.  
\end{lem}

\begin{proof}
This follows directly from the inclusion exclusion principle, or equivalently the identity $$\sum_{k=1}^{n} (-1)^{k+1} {r \choose k} = 1- (-1)^n {r-1 \choose n},$$ where ${r-1 \choose n}$ is defined to be zero if $n \geq r$. Every surface in $\cM_g^{<\e}$ has some number $r$ of geodesics of length at most $\e$, where $1\leq r \leq 3g-3$. 
\end{proof}

Let $n(g)$ be the floor of $\log g / \log \log \log g$. 
For each $i$, write $A_k = G_k + B_k$, where the good contribution $G_k$ is the integral of the number of sets $S$ of $k$ disjoint unoriented geodesics, each of length at most $\e$, where the complement of $S$ is connected, and the bad contribution $B_k$ is defined similarly for sets where $S$ is separating. 

\begin{lem}\label{L:ThinMeasureGood}
There is a constant $\e_0>0$ such that
for $\e<\e_0$,
$$\sum_{k=1}^{n(g)} (-1)^{k+1} G_k = V_g (1+O( g^{-1+o(1)})) (1-\exp(-I_\e)).$$
Moreover $G_{n(g)+1}=O(V_g  g^{-1+o(1)}(1-\exp(-I_\e))).$
\end{lem}

\begin{proof}
Mirzakhani's integration formula gives
$$G_k=\frac1{2^k k!} {\int_0^\e \cdots \int_0^\e} \delta_1 \cdots \delta_k V_{2g-k, 2k}(\delta_1, \delta_1,  \ldots, \delta_k, \delta_k)  d \delta_1 \ldots d \delta_k. $$
Since $k\leq n(g)$, the sinh approximation (Lemma \ref{L:sinh}) gives 
$$G_k =  (1+O(g^{-1+o(1)})) \frac{V_{g-k,2k}} { k!} I_\e^k  .$$
As in equation \eqref{E:aratio}, the volume asymptotics with error term in \cite[Theorem 1.8]{MirzakhaniZograf:LargeGenus} give that $V_{g-k,2k}$ is very close to $V_g$, so this implies the same statement with $V_{g-k,2k}$ replaced by $V_g$. 
Hence 
$$ \sum_{k=1}^{n(g)} (-1)^{k+1} G_k  = V_g \sum_{k=1}^{n(g)} \frac{(-1)^{k+1}}{k!} I_\e^k + O\left(V_g g^{-1+o(1)}\sum_{k=1}^{n(g)} \frac{1}{k!} I_\e^k\right).$$

Taylor's Theorem implies that 
$$\sum_{k=1}^{n(g)} \frac{(-1)^{k+1}}{k!} I_\e^k = 1-\exp(-I_\e) + O\left( \frac{(I_\e)^{n(g)+1}}{(n(g)+1)!} \right).$$
We think of $1-\exp(-I_\e)$ as being the main term, and note that it is approximately $I_\e$ when $\e$ is small. 
We need to compare the error here and above to this main term. We start by noting that 
$$\sum_{k=1}^{n(g)} \frac{1}{k!} I_\e^k \leq \exp(I_\e)-1.$$
Since this is also about $I_\e$ when $\e$ is small, we have 
$$V_g g^{-1+o(1)}\sum_{k=1}^{n(g)} \frac{1}{k!} I_\e^k= O\left( V_g g^{-1+o(1)} (1-\exp(-I_\e))   \right),$$
bounding with the first source of error above. 

Next we consider the error the Taylor approximation, namely 
$$
\frac{(I_\e)^{n(g)+1}}{(n(g)+1)!} \leq \frac{I_\e}{(n(g)+1)!} =  \frac{O(1-\exp(-I_\e))}{(n(g)+1)!} , 
$$
where we assume in particular that $\e$ is small enough to get $I_\e<1$. 
This error term is small enough when $(n(g)+1)!>g$.
Using Stirling's formula, this is certainly true when 
$$ (n(g)/e)^{n(g)} = e^{(\log(n(g))-1) n(g)} > g,$$
which is guaranteed by our choice of $n(g)$. 

The final statement follows from the arguments above, and can also be obtained by summing up to $n(g)+1$ instead of $n(g)$ and then subtracting the two sums to isolate the $k=n(g)+1$ term. 
\end{proof}

\begin{lem}\label{L:ThinMeasureBad}There is a constant $\e_0>0$ such that
for $\e<\e_0$,
 $$\sum_{k=1}^{n(g)+1} B_k \leq V_g \cdot g^{-1+o(1)} (1-\exp(-I_\e)).$$
\end{lem}

\begin{proof}
By Mirzakhani's Integration Formula and the sinh upper bound in Lemma \ref{L:sinh}, 
$B_k$
is at most $I_\e^k$ times a sum of products of $V_{g_i, n_i}$, summed over all ways of pinching $k$ curves to a get a surface with at least $2$ components, where the components have genus $g_i$ and $n_i$ nodes.  (One can equally well cut along each of the $k$ curves, in which case one obtains a disconnected surface with components having $n_i$ boundary circles instead of $n_i$ nodes.)   %Given such a way of pinching $k$ curves, we use $q\geq 2$ to denote the number of components produced, and let $q'\leq q$ denote the number of components that don't have $(g_i, n_i) \in \{(0,3), (1,1)\}.$

Lemma \ref{L:CountStrata} states that there are at most $2^{k+q^2} g^{q'-1}$ ways to pinch a collection of $k$ curves and get a nodal surface with $q$ components, $q'$ of which aren't spheres with three marked points or tori with one marked point. Lemma \ref{L:ProductBound} states that for each such configuration, 
$$\prod_{i=1}^q V_{g_i, n_i}
\leq      V_g     \left( \frac{C_1}{g}\right)^{q+q'-2}.$$
Note that 
$$2^{k+q^2} g^{q'-1} \left( \frac{C_1}{g}\right)^{q+q'-2} = \frac{2^{k+q^2} C_1^{q+q'-2}}{g^{q-1}} \leq  (2C_1^2)^{k-1} \left(\frac{ g^{o(1)} }{g}\right)^{q-1},$$ 
where the final inequality uses that  $q\leq k+1 \leq n(g)+2$ and hence  $2^q \leq 2^{n(g)+2} = g^{o(1)}$.  

For each $q$, there are at most $q$ values of $q'.$  Also, $q$ is at most $k+1.$  Thus,
$$ \sum_{k=1}^{n(g)+1} B_k \leq  V_g \sum_{k=1}^{n(g)+1} I_\e^k (2C_1^2)^{k-1}  \sum_{q=2}^{k+1} \left(\frac{ g^{o(1)} }{g}\right)^{q-1}. $$

If $\e_0$ is small enough, then $I_\e \cdot (2 C_1^2) < 1/2$, so this is bounded by 
 $$V_g I_\e \sum_{k=1}^{n(g)+1} \frac{1}{2^{k-1}} \sum_{q=2}^{n(g)+1} \left(\frac{ g^{o(1)} }{g}\right)^{q-1} =  V_g I_\e  g^{o(1)-1}.$$
Since $1-\exp(-I_\e)$ is comparable to $I_\e$, this gives the result. 
%old:
%$$g^{o(1)} \sum_{q=2}^{n(g)+1} q g^{-(q-1)} I_\e^q = V_g g^{o(1)-1} (1-\exp(-I_\e)),$$
%giving the lemma. 
\end{proof}

To conclude the proof of the proposition, first note that the alternating over/underestimate of the truncated inclusion exclusion bounds from Lemma \ref{L:IE} shows that the error in the truncated inclusion exclusion is at most $A_{n(g)+1}=G_{n(g)+1}+B_{n(g)+1}$. Lemma  \ref{L:ThinMeasureGood} bounds $G_{n(g)+1}$, and Lemma \ref{L:ThinMeasureBad} overestimates $B_{n(g)+1}$, since the sum goes up to $n(g)+1$. 

Hence Lemmas \ref{L:ThinMeasureGood} and \ref{L:ThinMeasureBad} give the proposition.
\end{proof}

\begin{prop}\label{P:ThinIntegral}
There is a constant $\e_0>0$ such that
for $\e<\e_0$ and $F$ a non-negative function with support in $[0, g^{o(1)}]$, the integral of $F_{\mathrm{sns}}$ over $\cM_g^{<\e}$ at least  
$$V_g(1-g^{-1+o(1)}) \cdot (1-\exp(-I_\e)) \cdot I_F .$$
\end{prop}

\begin{proof}
The proof is almost identical to the previous proposition, so we only give a sketch. Let $n(g)$ be an even integer closest to $\log g / \log \log \log g.$

Let $A_k'$ be the integral over $\cM_g$ of the sum over simple non-separating geodesics $\gamma$ on $X\in \cM_g$ of
$F(\ell_X(\gamma))$ times the number of sets  $S$ of $k$ disjoint geodesic of length at most $\e$ all of which are disjoint from $\gamma$. 

As in Lemma \ref{L:IE}, since $n(g)$ is even, the desired integral is bounded below by
$$\sum_{k=1}^{n(g)} (-1)^{k+1} A_k'.$$
Indeed, if $\gamma$ is a simple non-separating geodesic, let $r$ denote the number of geodesics of length at most $\e$ disjoint from $\gamma$. If $r=0$, $\gamma$ does not contribute to this sum, and if $r>0$ then, since $n(g)$ is even, the proof of Lemma \ref{L:IE} shows that $\gamma$ contributes at most once.   

For each $i$, decompose $A_k' = G_k' + B_k'$, where $G_k'$ is the contribution from $S\cup \gamma $  non-separating, and $B_k$ is contribution from $S\cup \gamma $ separating.

\begin{lem}\label{L:ThinIntegralGood}
There is a constant $\e_0>0$ such that
for $\e<\e_0$,
$$\sum_{k=1}^{n(g)} (-1)^{k+1} G_k' = V_g (1+g^{-1+o(1)})\cdot (1-\exp(-I_\e)) \cdot I_F. $$
Moreover $G_{n(g)+1}'=O(V_g  g^{-1+o(1)}\cdot (1-\exp(-I_\e)) \cdot I_F).$
\end{lem}

\begin{proof}
In this case, because of our assumption on the support of $F$, the sinh approximation gives  
$$G_k' =  (1-g^{-1+o(1)}) \frac{V_g} { k!} I_\e^k I_F,$$
 and otherwise the proof proceeds as in Lemma \ref{L:ThinMeasureGood}, since this expression for $G_k'$ is $I_F$ times the expression for $G_k$ that appeared in Lemma \ref{L:ThinMeasureGood}.
\end{proof}

\begin{lem}\label{L:ThinIntegralBad}
$$\sum_{k=1}^{n(g)+1} B_k' \leq V_g (g^{-1+o(1)}) \cdot(1-\exp(-I_\e))\cdot I_F.$$
\end{lem}

\begin{proof}
The proof is similar to Lemma \ref{L:ThinMeasureBad}.
\end{proof}

The proposition follows from Lemmas \ref{L:ThinIntegralGood} and \ref{L:ThinIntegralBad}.  
\end{proof}

\section{Proof of Theorem \ref{T:GeometricMain}}\label{S:GeometricMain}

Fix $\kappa$ and $D$, and consider the locus $\cN_g\subset \cM_g$  where 
\begin{enumerate}
\item there are no separating multi-curves of total length less than $(\kappa/2) \log(g)$ whose complement has two components, and 
\item there are no separating multi-curves of total length less than $2D\log(g)$ whose complement has two components, each of area at least  $2\pi (4D+1)$. 
\end{enumerate}

%We observe that the complement of $\cN_g$ is small. 

\begin{lem}\label{L:NgMeasure} 
The measure of the complement of $\cN_g$ is 
$$O( g^{\kappa-1} V_g    ). $$
\end{lem}

\begin{proof}
 Corollary \ref{C:SepVol} states that, for integer $a\geq 0$, the probability that a surface in $\cM_g$ has a multi-geodesic of length at most $L$ cutting the surface into two components each area at least $2\pi a$  is 
$$O(e^{2L} \cdot g^{-a}) .$$
Thus, the probability of having a separating multi-curve of total length less than $(\kappa/2) \log(g)$ is 
$$O( g^\kappa \cdot g^{-1})$$
and the probability of having a separating multi-curve of total length less than $2D\log(g)$ cutting the surface into two components each area at least $2\pi (4D+1)$  is 
$$O( g^{4D} \cdot g^{-(4D+1)}).$$
This proves the lemma.
\end{proof}

Define $\cN_g^{>\e}= \cN_g \cap \cM_g^{>\e}$. For now we require only that $\e$ is smaller than some universal constant, but ultimately we will take $\e$ to zero as $g\to \infty$. Throughout this section, we assume $F$ is supported on $[0, D\log(g)]$. 

\begin{lem}\label{L:snsMain}
The average over $\cN_g^{>\e}$ of $F_{\mathrm{sns}}$ is at most 
$$\left(1+ O(g^{\kappa-1}) \right) I_F.$$
\end{lem}

\begin{proof}
Compute 
\begin{eqnarray*}
\frac{1}{\Vol(\cN_g^{>\e})}\int_{\cN_g^{>\e}} F_{\mathrm{sns}} &=& \frac{\Vol(\cM_g^{>\e})}{\Vol(\cN_g^{>\e})} \cdot \frac{1}{\Vol(\cM_g^{>\e})} \int_{\cN_g^{>\e}} F_{\mathrm{sns}}
\\&\leq& \frac{\Vol(\cM_g^{>\e})}{\Vol(\cN_g^{>\e})} \cdot \frac{1}{\Vol(\cM_g^{>\e})} \int_{\cM_g^{>\e}} F_{\mathrm{sns}}.
\end{eqnarray*}
We will separately give bounds in the first and second factor of this expression. 

Corollary \ref{C:ThickAverage} states that the second factor is at most $(1+g^{-1+o(1)}) I_F$. Note that 
\begin{eqnarray*}
\frac{\Vol(\cM_g^{>\e})}{\Vol(\cN_g^{>\e})}  &=& 1 + \frac{\Vol(\cM_g^{>\e}\setminus \cN_g^{>\e} )}{\Vol(\cN_g^{>\e})}
\\ & \leq &
1 + \frac{\Vol(\cM_g\setminus \cN_g )}{\Vol(\cN_g^{>\e})}
\\ & \leq &
1 + \frac{2\Vol(\cM_g\setminus \cN_g )}{V_g},
\end{eqnarray*}
where in the last line we used the extremely weak bound $\Vol(\cN_g^{>\e}) \geq V_g/2$. So Lemma \ref{L:NgMeasure} gives that the first factor is at most $1+ O( g^{\kappa-1})$. 
\end{proof}

\begin{prop}\label{P:nsnsMain}
The average over $\cN_g^{>\e}$ of $F_{\mathrm{all}}-F_{\mathrm{sns}}$  is at most 
$$O( g^{o(1)-1} I_F).$$
\end{prop}

A union of closed geodesics is said to fill a hyperbolic surface if every component of the complement is either a contractible polygon or an annular region around a boundary geodesic. Recall the following. 

\begin{lem}\label{L:fill}
Suppose a union $\gamma$ of closed geodesics of total length $\ell$ fills a hyperbolic surface $X$ of area $A$ with boundary of length $B\geq 0$.  Then $B< 2\ell$ and 
$\ell > A/4$.
\end{lem}

\begin{proof}
Each boundary geodesic can be obtained by tightening a path of segments of $\gamma$, and each segment can contribute at most twice in this way. So $B<2\ell$. 

A version of the isoperimetric inequality gives that, for each component of the complement of $\gamma$, the length of the boundary of this component is greater than the area \cite[page 211]{Buser}. So $2\ell+B >A.$
\end{proof}

\begin{cor}\label{C:subsurface} 
If $g$ is larger than a constant depending on $D$, then any non-simple geodesic of length at most $D\log(g)$ on a surface in $\cN_g$ is contained in a subsurface with boundary at most $2D\log(g)$ and area at most $2\pi (4D+1)$ and with connected complement. 
% \textcolor{red}{I'm confused about the current statement.  Imagine our reasonably short geodesic fills a surface of genus 0 with 10 boundary components and that the complete surface has some genus $g/10$ chunk sticking off of every boundary component (so the surface looks like a ``star with 10 equally sized spikes").  Does that geodesic sit inside a subsurface with small area and small boundary and with connected complement?  As soon as you contain that core genus 0 subsurface, it seems like you need to contain tons of other stuff so as not to disconnect the complement.}\textcolor{blue}{ The definition of $\cN_g$ causes some magic.}
\end{cor}

\begin{proof}
By Lemma \ref{L:fill}, any non-simple geodesic of length at most $D\log(g)$ fills a subsurface $S$ with boundary of length at most $2D\log(g)$ and area at most $4D\log(g)$. A surface of that area has at most $2D\log(g)/\pi$ boundary circles, so the complement of $S$ can have at most that many components. Let $C$ be the component of the complement of $S$ with largest area, so $C$ has area at least 
$$ \frac{2\pi(2g-2)- 4D\log(g)}{2D\log(g)/\pi}.$$
Assume $g$ is large enough so that this quantity is greater than $2\pi (4D+1)$. 

Let $S'$ be the complement of $C$. Note that $S'$ is connected, because it contains $S$, which is adjacent to every component of the complement of $S$.

By the second condition in the definition of $\cN_g$, we see that $S'$ must have area at most $2\pi (4D+1)$, since its complement is connected and area greater than $2\pi (4D+1)$. 

Since the geodesic is contained in $S'$, this gives the result. 
\end{proof}

\begin{proof}[Proof of Proposition \ref{P:nsnsMain}]
We will use that $F$ has support in $[0, D\log(g)]$. 

Since $D$ is fixed, there are only a finite number of possible topological types for a subsurface of area at most $2\pi (4D+1)$. Thus Corollary \ref{C:subsurface} motivates the following.

\begin{lem}\label{L:g1k}
For fixed $g_1$ and $k$,  if $\e$ is such that $1/\e$ is $g^{o(1)}$, then 
the average over $\cN_g^{>\e}$ of the sum of $F$   evaluated at the lengths of the  geodesics contained in a subsurface of genus $g_1$ with $k$ boundary components and connected complement is at most
$$  O( g^{-1+o(1)} I_{\widetilde{F}}).$$
\end{lem}

\begin{proof}
We estimate the average number of such geodesics of length at most $L$, assuming $L\leq D\log(g)$. Each such geodesic is contained in a subsurface of boundary at most $2L$. 

Corollary \ref{C:SepVol} (with $a=2g_1-2+k$ fixed) gives that the average number of such subsurfaces is at most  $$O(e^{(2L)/2} (2L)^{p} g^{-1})$$  for some $p$. Note that, since the volume of $\cN_g^{>\e}$ is certainly at least half that of $\cM_g$, the average over $\cN_g^{>\e}$ is at most twice the average over $\cM_g$.

Theorem \ref{T:GraphBound} gives that the number of geodesics of length at most $L$ in each such subsurface is at most 
$$\left(\frac{C L_0 \log(1/\e)}{\e}\right)^{C L/L_0+3}$$
where $L_0=(\kappa/2)\log(g)$. Given $L\leq D\log(g)$, the exponent $C L/L_0+3$ is $O(1).$  Given the restriction on $\e$, this whole expression is $g^{o(1)}$, with little $o$ function depending on $D$ and $\kappa$. 

So the average over $\cN_g^{>\e}$ of the number of  geodesics of length at most $L\leq D\log(g)$ contained in a subsurface of genus $g_1$ with $k$ boundary components is 
$$O(e^{L} L^{p} g^{-1+o(1)}).$$
Given the bound on $L$ this is 
 $$O(g^{-1+o(1)} L \sinh^2(L/2) / (L/2)^2).$$

Using that the number of such geodesics of length in $[n-1, n]$ is at most the number of such geodesics of length at most $n$, we get that the average in question is bounded by a constant times
 $$\sum_{n=1}^\infty \left(\sup_{[n-1,n]} F \right) \cdot \left(g^{-1+o(1)} n \sinh^2(n/2) / (n/2)^2 \right).$$
The result follows since $\widetilde{F}(\ell) \geq \sup_{[n-1,n]} F$ for any $\ell\in [n-1,n]$, and since $n \sinh^2(n/2) / (n/2)^2$ is at most a constant not depending on $n$ times $\int_{n-1}^n x \sinh^2(x/2) / (x/2)^2 dx$.

% \color{red}
% Partition $[0,\infty]$ into the intervals $[n,n+1]$ for $n \in \mathbb{Z}_{\geq 0}.$  Let $G(n)$ be the supremum of $F$ on the interval $[n,n+1].$  Since $\sum G(n) \mathbf{1}_{[n,n+1]},$ and hence $\sum G(n) \mathbf{1}_{[0,n]},$ is larger than $F,$ the average over $\cN_g^{> \e}$ of $F$ summed over the above sorts of geodesics is at most 
% \begin{align*}
%     &\sum_n G(n)  O \left(g^{-1+o(1)} n \sinh^2(n/2) / (n/2)^2 \right) \\
%     &= O \left( g^{-1+o(1)} \sum_n G(n)   n \sinh^2(n/2) / (n/2)^2 \right). 
% \end{align*}
%
% Since $\int_n^{n+1} F(x) dx$ is the average value of $F$ over $[n,n+1],$ we can bound
% $$\left| G(n) - \int_n^{n+1} F(x) dx  \right| \leq \mathrm{BV}_{[n,n+1]}(F),$$
% where $\mathrm{BV}$ denotes the variation norm.  The above is thus bounded by
% \begin{align*}
%     &=  O \left( g^{-1+o(1)} \sum_n BV_{[n,n+1]}(F) n \sinh^2(n/2) / (n/2)^2   + \int_n^{n+1} F(x) dx \; n \sinh^2(n/2) / (n/2)^2 \right) \\
%     &= O \left( g^{-1+o(1)} \sum_n BV_{[n,n+1]}(F) n \sinh^2(n/2) / (n/2)^2   + \int_n^{n+1} F(x) x \sinh^2(x/2) / (x/2)^2 dx \right) \\
%     &= O \left( g^{-1+o(1)} \sum_n BV_{[n,n+1]}(F) n \sinh^2(n/2) / (n/2)^2   + I_F \right).
% \end{align*}
% \color{black}
% % 
% Integrating this against $F(L)$ gives the result, since the probability density function for the number of geodesics in question of length exactly $L$ is certainly bounded by the probability density function for geodesics of length at most $L$. 
\end{proof}

\begin{lem}\label{L:SepAvg}
The average over $\cN_g^{>\e}$ of the sum of $F$ over simple separating geodesics is at most 
$$O( g^{o(1)-1} I_{ \widetilde{F} }).$$
\end{lem}
\begin{proof}
First consider the average number of separating geodesics of length at most $L$, averaged over $\cN_g^{>\e}$. Since the volume of $\cN_g^{>\e}$ is certainly at least half that of $\cM_g$, this is at most twice the average over $\cM_g$. 

Corollary \ref{C:SepVol} (with $k=1$) gives that the average over $\cM_g$ is $O(e^{L} L^{2} g^{-1})$. Assuming $L\leq D\log(g)$, this is bounded by a constant times $$g^{-1+o(1)} L \sinh^2(L/2) / (L/2)^2.$$
 The result follows as in the previous lemma. 
%Integrating this against $F(L)$ gives the result.
\end{proof}

The two lemmas prove the proposition, since every geodesics contributing to $F_{\mathrm{all}}-F_{\mathrm{sns}}$ is either simple and separating, and hence controlled by Lemma \ref{L:SepAvg}, or contained in a subsurface of one of finitely many topological types by Corollary \ref{C:subsurface}, and hence controlled by Lemma \ref{L:g1k}. 
\end{proof}

\begin{proof}[Proof of Theorem \ref{T:GeometricMain}]
Set $\e=1/\log(g)$ and define $\cM_g' = \cN_g^{>\e}$. Since $\e\to 0$ as $g\to \infty$, the probability measure of $\cM_g^{<\e}$ goes to zero as $g\to\infty$. %, by Proposition \ref{P:ThinMeasure} or the weaker results preceding it. 
Lemma \ref{L:NgMeasure} gives that the probability measure of the complement of $\cN_g$ goes to zero as $g\to \infty$. So 
${\Vol(\cM_g')}/{V_g} \to 1.$

Lemma \ref{L:snsMain} and Proposition \ref{P:nsnsMain} give the estimate on the integral of $F_{\mathrm{all}}$. 
\end{proof}

\section{Proof of Theorem \ref{T:SpectralMain}}\label{S:SpectralMain}

In this section we prove Theorem \ref{T:SpectralMain} by averaging Selberg's trace formula \cite{Selberg}. 

%Suppose $\mu$ is an arbitrary Borel probability measure on $\cM_g$ satisfying the thick support and averaging hypotheses of Theorem \ref{T:SpectralMain}.  The purpose of this section is to explain how, by averaging the Selberg trace formula with respect to $\mu,$ we may prove $\lambda_1$ is greater than $\frac{3}{16} - \delta$ with high $\mu$-probability.  %We argue in this generality in order to disentangle the trace formula input needed in this paper from random WP-geometry. 

\subsection{Statement of the trace formula} For smooth, even, compactly supported functions $f$ on $\mathbb{R},$ define 
$$F_f(x) = x \cdot \sum_{k = 1}^\infty \frac{f(kx)}{2 \sinh(kx/2)}.$$
We continue to use $F_{f,\mathrm{all}}(X)$ to denote the sum of $F_f$ over the lengths of primitive oriented closed geodesics on $X$. 

\begin{thm}\label{T:TF}%[Selberg trace formula \cite{???}]
Let $f$ be a smooth, even, compactly supported function on $\mathbb{R}.$  Let $X$ be a closed hyperbolic surface of genus $g$ with Laplace eigenvalues $\lambda_n = \frac{1}{4} + r_n^2$. Then

$$\sum_{r_n } \widehat{f}(r_n) 
= (g-1)  \int_{-\infty}^{\infty} \widehat{f}(r) \cdot r \cdot \tanh(\pi r) \; dr + F_{f,\mathrm{all}}(X).$$
\end{thm}

The left hand side is called the \emph{spectral side}, and the right hand side the \emph{geometric side}. The first summand on the geometric side is called the \emph{identity contribution}. The imaginary parameters $r_n,$ corresponding to eigenvalues strictly less than $\frac{1}{4},$ are called \emph{exceptional}.

\medskip

Since $f$ is even, its Fourier transform equals
$$\widehat{f}(r) = \int_{\mathbb{R}} f(x) e^{-ir \cdot x} dx = 2\int_{0}^\infty f(x) \cosh(-ir x) dx.$$

\subsection{A preliminary observation.}
We start by noting that the integral $I_{F_f}$ is close to the $\lambda_0 = 0$ contribution to the trace formula. 

\begin{lem}\label{L:IFf}
$|I_{F_f}- \widehat{f}( i/2 )| \leq 4 \|f\|_{1}.$
\end{lem}

\begin{proof}
Directly from the definitions, we get  
\begin{eqnarray*}
I_{F_f} &=& 2 \sum_{k = 1}^\infty \int_0^\infty  \frac{f(k\ell) \sinh(\ell/2)^2}{\sinh(k\ell/2)} d \ell
\\&=&  2\int_0^\infty f(\ell) \sinh(\ell/2) d \ell + 
2 \sum_{k = 2}^\infty \int_0^\infty  \frac{f(k\ell) \sinh(\ell/2)^2}{\sinh(k\ell/2)} d \ell.
\\&=& \widehat{f} \left( i/2 \right) - 2\int_0^\infty f(\ell) e^{-\ell/2} d\ell+ 
2 \sum_{k = 2}^\infty \int_0^\infty  \frac{f(k\ell) \sinh(\ell/2)^2}{\sinh(k\ell/2)} d \ell.
\end{eqnarray*}
Since $e^{-\ell/2}\leq 1$, the middle term is at most $2\|f\|_1$. 

\begin{comment}
When $k\geq 2$, convexity of $\sinh$ gives  
$$\frac{\sinh(\ell/2)^2}{\sinh(k\ell/2)} \leq \frac{\sinh(\ell/2)^2}{\frac{k}2 \sinh(\ell)} \leq \frac1{k}.$$
\end{comment}

When $k\geq 2$, convexity of sinh gives
\begin{align*}
\frac{\sinh(\ell/2)^2}{\sinh(k\ell/2)} &\leq \frac{\sinh(\ell/2)^2}{\frac{k}2 \sinh(\ell)} 
\leq \frac1{k}. 
\end{align*}
%
% DO NOT DELETE (AW July 10): 
%
% The first inequality is convexity. (Mike says you can also get it with power series.) You can get the second inequality via the double angle formula for sinh
% and the fact that tanh <= 1
%

Applying the change of variables $u=k\ell$ and noting that $2\sum_{k=2}^\infty \frac1{k^2}< 2$, 
this gives that the third term is at most $2\|f\|_1$.
%Since $\sinh(\ell/2)^2/\sinh(\ell) \leq \frac12$, the $k=2$ term of the sum is bounded by $\|f\|_1/2.$
%
% The function $F_f$ has only a logarithmic singularity at $0.$\footnote{If $f$ is supported on $[-R,R],$ then
% $$F_f(x) \leq \|f\|_{\infty} \cdot x \cdot \sum_{k = 1}^{R/x} \frac{1}{2\sinh(kx/2)} \leq \sum_{k = 1}^{R/x} \frac{1}{2k} = \frac{1}{2} \log(R/x) + O(1).$$} For later use, we record: 
%
% \begin{align} \label{leadingterm}
% I_{F_f} &:= \int_0^\infty F_f(\ell) \cdot \ell \cdot \left( \frac{\sinh(\ell/2)}{\ell/2} \right)^2 d \ell  \\
% &= \int_0^\infty \ell \cdot \frac{f(1 \cdot \ell)}{\sinh(1 \cdot \ell/2)} \cdot \ell \cdot \left( \frac{\sinh(\ell/2)}{\ell/2} \right)^2 d \ell + \int_0^\infty \ell \cdot \sum_{k = 2}^\infty \frac{f(k\ell)}{\sinh(k\ell/2)}  \cdot \ell \cdot \left( \frac{\sinh(\ell/2)}{\ell/2} \right)^2 d \ell \nonumber \\ %
% &= \frac{1}{2} \widehat{f} \left( i/2 \right) + O( \|f\|_{\infty} ). \nonumber \\ \nonumber
% \end{align}
%
\end{proof}

\begin{cor}\label{C:Cancel} 
Under the assumptions of Theorem \ref{T:SpectralMain}, if $f$ is even and has support in $[-D\log(g) ,D\log(g)]$,  we have the one-sided bound
$$ \int_{\cM_g}\left(F_{f,\mathrm{all}}(X)- \widehat{f}\left( i/2 \right)\right) d \mu(X) \leq 4 \|f\|_{1}+  O(g^{\kappa-1})I_{\widetilde{F}_f}.$$
\end{cor} 

% \textcolor{blue}{Below calculation can be deleted or suppressed.  Transcription purposes.}
% \color{red}
% \begin{proof}
%     The integral is at most
%     \begin{align*}
%         I_{F_f} - \widehat{f}(i/2) + O\left( g^{\kappa - 1} I_{\widetilde{F}_f} \right) &\leq 4 ||f||_1 +  O\left( g^{\kappa - 1} I_{\widetilde{F}_f} \right) \\
%         &= 4 ||f||_1 +  O\left( g^{\kappa - 1} \right) I_{\widetilde{F}_f}.
%     \end{align*}
% \end{proof}
% \color{black}

This cancellation on average is the essential point in our arguments below. 

\subsection{Picking test functions}
% Lemma \ref{L:IFf} and hypothesis (2) from the statement of Theorem \ref{S:SpectralMain} $F_{f,all}(X)$. 

% \begin{rem}
% Per \eqref{leadingterm}, $I_{F_f}$ essentially equals the $\lambda = 0$ contribution to the trace formula for the test function $f.$  Hypothesis (2), from the statement of Theorem \ref{S:SpectralMain}, yields non-trivial cancellation between the latter constant contribution to the trace formula - the ``leading term" in our application - and the $\mu$-average of the trace formula.  This cancellation on average is the essential point of the below proof.
% \end{rem}

Fix   a smooth, compactly supported, even test function $f$ on $\mathbb{R}$ satisfying
\begin{itemize}
\item
$f$ is non-negative and supported on $[-1,1]$ and
\item
$\widehat{f} \geq 0$ on $\mathbb{R} \cup i \mathbb{R}$ with $\widehat{f} > 0$ on $i \mathbb{R}.$
\end{itemize}
For example, $f$ could be be the convolution square of a smooth, even, non-negative function $g$ supported on $[-1/2,1/2]$ with $g(0) > 0.$

Define $$f_L = \frac{1}{2}( f(x + L) + f(x - L) ).$$  The Fourier transform intertwines translation and multiplication by characters, so $\widehat{f_L}(r) = \widehat{f}(r) \cdot \cos(Lr)$ and $\widehat{f_L}(it) = \widehat{f}(it) \cdot \cosh(Lt).$ We will assume that $L\leq D\log(g)-1$. 

Our goal is to give an upper bound for
$$p = \mu \left( \left\{ X \in \cM_g: \lambda_1(X) \leq \frac{1}{4} - b^2 \right\} \right).$$
We start by relating this to the contribution of exceptional eigenvalues. 

\begin{lem}\label{L:AVGlower}
The $\mu$-average of
$$\sum_{r_n \in (0 \cdot i, \frac{1}{2} \cdot i) } \widehat{f_L}(r_n(X)) $$  is at least $p \cdot \cosh(Lb) \cdot m,$
where $m = \min_{t \in [0,1/2]} \widehat{f}(it)$.
\end{lem}

\begin{proof}
This follows immediately from monotonicity of $\cosh$ and the non-negativity property of $\widehat{f}$.
\end{proof}

In the remainder of the proof, we use the trace formula to give an upper bound for this average, which will translate into an upper bound for $p$. 

\begin{lem}\label{L:AVGupper}
The $\mu$-average of
$$\sum_{r_n \in (0 \cdot i, \frac{1}{2} \cdot i) } \widehat{f_L}(r_n(X)) $$ 
is less than or equal to 
$O \left( e^{L/2} \cdot g^{\kappa - 1} + g^{1 + o(1)} \right).$ 
\end{lem}

\begin{proof}
The trace formula allows us to write $\sum_{r_n \in (0 \cdot i, \frac{1}{2} \cdot i) } \widehat{f_L}(r_n)$ as
$$ (g-1) \int_{-\infty}^{\infty} \widehat{f_L}(r) \cdot r \cdot \tanh(\pi r) \; dr - \sum_{r_n \text{ real}} \widehat{f_L}(r_n) 
+ \left( F_{f_L,\mathrm{all}}(X) - \widehat{f_L}(i/2) \right). $$
%The second term is negative, and we will show the first term is small
We will show that the first term is small, the second term is small for all $X$ in the support of $\mu,$ and that the $\mu$-average of the third term is small. 

To start, note that since $\widehat{f_L}(r) = \widehat{f}(r)  \cos(Lr),$ the first term is bounded by 
$$(g-1)  \int_{-\infty}^{\infty} |\widehat{f}(r)|  r \; dr = O(  g ).$$ 

In  Corollary \ref{C:AppendixC}, for fixed $h$, we show using a standard local Weyl law argument that, for all $X\in \cM_g$, 
\begin{align*}
 \sum_{r_n \text{ real}}|h(r_n)| &= O \left(g \cdot \logp\left(\frac{1}{\mathrm{sys}(X)}\right) \right),
\end{align*}
where $\logp(x) = \max \{0, \log(x) \}+1$.  For $X$ in the support of $\mu$, keeping in mind that $|\widehat{f_L}| \leq |\widehat{f}|$, this gives the bound 
$$ \sum_{r_n \text{ real}} | \widehat{f_L}(r_n(X)) | =  O\left(g^{1 + o(1)} \right)$$
for the second term above. %, since $\|\widehat{f_L}\|_W $ is bounded in terms of the constant $\|\widehat{f}\|_W$. 

 A straightforward estimate shows that $\widetilde{F}_{f_{L}}(\ell)$ is $O((\ell+1) e^{-L/2})$, and hence $I_{\widetilde{F}_{f_L}}$ is $O( e^{L/2})$. Thus 
 Corollary \ref{C:Cancel} shows that the $\mu$-average of the third term is at most 
\begin{align*}
    4 \|f_L\|_{1}+  O(g^{\kappa-1}) I_{\widetilde{F}_{f_L}} &= 4 \|f_L\|_{1}+  O(g^{\kappa-1}) O( e^{L/2}) \\  
    &= O(1  + g^{\kappa-1}  e^{\frac{L}2}). 
    %&= O(1+ \textcolor{red}{L} + g^{\kappa-1} e^{\frac{L}2}). 
\end{align*} 

Combining the bounds for the three terms gives the lemma. 
\end{proof}

We can now conclude the proof. 

\begin{proof}[Proof of Theorem \ref{T:SpectralMain}]
Combining the upper bound from Lemma \ref{L:AVGupper} with the lower bound from Lemma \ref{L:AVGlower} yields 
\begin{equation*}
p = O \left( e^{(\frac{1}{2} - b) L} \cdot g^{\kappa - 1} +  g^{1 + o(1)} \cdot e^{-Lb} \right).
\end{equation*}

The two summands here are equal when $L$ equals $L_0 = (4 - 2\kappa + o(1)) \log g.$  For this particular choice of $L_0$ we get

\begin{equation*}
p = O \left(  g^{ 1 - 4b \left(1 - \frac{\kappa}{2} \right) + o(1)} \right),
\end{equation*}
proving the theorem. 
\end{proof}

\appendix
\section{Volume polynomials}\label{A:VolPoly}

Recall the standard notation for volume polynomials 
$$V_{g,n}(2 \mathbf{L}) = \sum_{|\mathbf{d}|\leq 3g-3+n} [\tau_{d_1} \cdots \tau_{d_n}]_{g,n} \frac{L_1^{2d_1}}{(2d_1+1)!} \cdots  \frac{L_n^{2d_n}}{(2d_n+1)!},$$
where $\mathbf{L}=(L_1, \ldots, L_n)$, $\mathbf{d}=(d_1, \ldots, d_n)$ and  $|\mathbf{d}| = \sum_{j=1}^n d_j$,

The sole purpose of this appendix is to check the following statement, which will only be used in the proof of Lemma \ref{L:sinh}.

\begin{thm}[Mirzakhani, Mirzakhani-Zograf]\label{T:Coeffs}
When $n=o(\sqrt{g})$,  
$$0\leq 1- \frac{[\tau_{d_1} \cdots \tau_{d_n}]_{g,n}}{V_{g,n}} \leq \frac{C n |\mathbf{d}|^2}{2g-3+n}.$$
\end{thm}

% \textcolor{red}{\begin{thm}[Mirzakhani, Mirzakhani-Zograf]\label{T:CoeffsModified}
% When $n=o(\sqrt{g})$, there is an absolute constant $C > 0$ for which  
% $$0\leq 1- \frac{[\tau_{d_1} \cdots \tau_{d_n}]_{g,n}}{V_{g,n}} \leq C \cdot \left( \frac{ |\mathbf{d}|^2 + n |\mathbf{d}| }{2g-3+n} + \frac{n^2 |\mathbf{d}|^2}{(2g-3 + n)^2} \right).$$\todo{A June 23: I made one denominator a square instead of the product of two almost equal things.}\todo{A June 23: Doesn't the first dominate the second? If so, maybe we should change the statement (and comment on this in the proof). M: Ok, and even ok if we want to leave the statement the way it was, maybe we should just switch back (but leave better version commented out)}
% \end{thm}}

More formally, this means that for any function $f(g) \in o(\sqrt{g})$ there is a constant $C$, so that this bound holds when $n\leq f(g)$. 
In fact the lower bound of $0$, which is easier, holds for all $g$ and $n$. 

% \textcolor{red}{Note that, since $|\mathbf{d}|\leq 3g-3+n$ and $n=o(\sqrt{g})$,
% $$C \cdot \left( \frac{ |\mathbf{d}|^2 + n |\mathbf{d}| }{2g-3+n} + \frac{n^2 |\mathbf{d}|^2}{(2g-3 + n)^2} \right) \leq \frac{C' n |\mathbf{d}|^2}{2g-3+n}$$
% giving a simpler but weaker upper bound.} \todo{A June 23: I added a comment saying that our new bound is better than our old one. Normally this would be odd to note this, but since the statement has changed I feel like it is worthwhile. (Keep in mind that $n>0$ in this section.)}  

For fixed $n$, this theorem is on \cite[page 286]{Mirzakhani:Growth} and also appears as \cite[Lemma 2.1]{MirzakhaniPetri:Lengths}, and the general statement is closely related to many statements in \cite{MirzakhaniZograf:LargeGenus}  such as \cite[Remark 3.2, Lemma 5.1]{MirzakhaniZograf:LargeGenus}. We  include a proof sketch since we could not find a precise statement to cite with the above level of uniformity.

\subsection{Recursion.} \label{recursion} Recall  the recursion  
$$[\tau_{d_1} \cdots \tau_{d_n}]_{g,n}  = \left(\sum_{j=2}^n \cA_\mathbf{d}^j \right)+ \cB_\mathbf{d} + \cC_\mathbf{d}$$
where we set  $d_0 = 3g-3+n - |\mathbf{d}|$ and

\begin{eqnarray*}
\cA_\mathbf{d}^j &=& 8 \sum_{d_1+d_j-1\leq k \leq d_0+d_1+d_j-1}(2d_j+1) 
 a_{k-d_1-d_j+1} [\tau_{k} \prod_{i \neq 1, j} \tau_{d_i}]_{g,n-1}
 \\[10pt]
 \cB_\mathbf{d} 
 &=& 16 \sum_{{d_1-2 \leq k_1+k_2\leq d_0+d_1 -2}} 
 a_{k_1+k_2-d_1+2}[\tau_{k_1}\tau_{k_2} \prod_{i\neq 1} \tau_{d_i}]_{g-1,n+1}
 \\[10pt]
  \cC_\mathbf{d} 
 &=& 16 \sum_{\substack{g_1+g_2=g\\ I \sqcup J = \{2, \ldots, n\} \\ d_1-2\leq k_1+k_2\leq d_0+d_1 -2 }} 
 a_{k_1+k_2-d_1+2}  [\tau_{k_1} \prod_{i\in I} \tau_{d_i}]_{g_1, |I|+1}  
 [\tau_{k_2} \prod_{i\in J} \tau_{d_i}]_{g_2, |J|+1}
 \end{eqnarray*}
 
\noindent This appears for example in \cite[Section 3.1]{Mirzakhani:Growth} and \cite[Equation 2.13]{MirzakhaniZograf:LargeGenus}. Here $$a_w = (1-2^{1-2w}) \zeta(2w),$$ which is a positive sequence that increases monotonically to the limit 1 for which there is a constant $c_0>0$ such that $w<w'$ implies $a_{w'} - a_w< c_0/2^{2w}$ \cite[Lemma 3.1]{Mirzakhani:Growth}. 
 
 \subsection{Upper bound.} We start with the easier inequality, which is a warm up for the harder one.
 
 \begin{lem}\label{L:UpperB}
 If $d_i' \leq d_i$ for all $i$ then 
 $$ [\tau_{d_1} \cdots \tau_{d_n}]_{g,n}  \leq [\tau_{d_1'} \cdots \tau_{d_n'}]_{g,n}.$$
 \end{lem}
 
 Since $V_{g,n}=[\tau_{0}^n]_{g,n}$, this immediately gives the following. 
 
 \begin{cor}\label{C:UpperB}
 $[\tau_{d_1} \cdots \tau_{d_n}]_{g,n} \leq V_{g,n}$.
 \end{cor}
 
 \begin{proof}[Proof of Lemma]
 By symmetry, it suffices to prove this when $d_i'=d_i$ for $i>1$, and this is what we will do. %(This helps with the $\cA_j$, $j>1$ terms, which have a factor of $(2d_j+1)$.)
So suppose $\mathbf{d'} = (d_1', d_2, \ldots, d_n)$ and $d_0' =3g-3+n - |\mathbf{d}'|$.

Since $d_1'\leq d_1$ and $d_0+d_1 = d_0'+d_1'$, we see that in the expressions above for $\cA_\mathbf{d}^j ,  \cB_\mathbf{d},   \cC_\mathbf{d} $,  the region of summation does not decrease when we pass from $\mathbf{d}$ to $\mathbf{d'}$. 
Monotonicity of the sequence $a_w$ thus gives $\cA_\mathbf{d}^j \leq \cA_\mathbf{d'}^j$, $\cB_\mathbf{d} \leq \cB_\mathbf{d'}$ and $\cC_\mathbf{d} \leq \cC_\mathbf{d'}$.
\end{proof}

\subsection{Error term.} We now give a lower bound for $[\tau_{d_1} \cdots \tau_{d_n}]_{g,n}$, by analyzing the error in the argument above. That error has two types, namely that from changing bounds in sums and that from changing values of $a_w$. 

\begin{comment}
\begin{lem} \label{dzerooneatatime}
Suppose that $n=o(\sqrt{g})$. Then there exists a constant $C>0$ such that, if
 $\mathbf{d'} = (0, d_2, \ldots, d_n)$, then 
$$[\tau_{d_1'} \cdots \tau_{d_n'}]_{g,n} -  [\tau_{d_1} \cdots \tau_{d_n}]_{g,n} \leq \frac{C |\mathbf{d}|^2 V_{g,n}}{2g-3+n}.$$
\end{lem}

\textcolor{red}{
\begin{lem}[Total incremental bound] \label{dzerooneatatimemodified}
Suppose that $n=o(\sqrt{g})$. Then there exists a constant $C>0$ such that, if
 $\mathbf{d'} = (0, d_2, \ldots, d_n)$, then 
\begin{align*}
[\tau_{d_1'} \cdots \tau_{d_n'}]_{g,n} -  [\tau_{d_1} \cdots \tau_{d_n}]_{g,n} &\leq \cA_\mathbf{d'} - \cA_\mathbf{d} \leq   \frac{ C \left( (n-1) (d_1 + 2) + 4 |\mathbf{d}| +  \sum_{j = 2}^n d_1 d_j  \right)  V_{g,n}}{2g-3+n} \\
&+ B \text{-bound} \\
&+ C \text{-bound}.
\end{align*}
\end{lem}
}
\end{comment}

\begin{lem}[Incremental $\cA$-term bound] \label{incrementalAbound}
Suppose that $n=o(\sqrt{g})$. Then there exists a constant $C>0$ such that, if
 $\mathbf{d'} = (0, d_2, \ldots, d_n)$, then 
\begin{equation*}
 \cA_\mathbf{d'} - \cA_\mathbf{d} \leq   \frac{ C \left( (n-1) (d_1 + 2) + 4 |\mathbf{d}| +  \sum_{j = 2}^n d_1 d_j  \right)  V_{g,n}}{2g-3+n},
\end{equation*}
where $\cA_\mathbf{d}:=\sum_{j=2}^n \cA_\mathbf{d}^j$. 
\end{lem}

\begin{proof}
We compute 
\begin{eqnarray*}
&&(\cA_\mathbf{d'}^j - \cA_\mathbf{d}^j)/8 
\\&=& 
 \sum_{d_1+d_j-1\leq k \leq d_0+d_1+d_j-1}(2d_j+1) 
(a_{k-d_j+1}- a_{k-d_1-d_j+1}) [\tau_{k} \prod_{i \neq 1, j} \tau_{d_i}]_{g,n-1} 
\\&& + \sum_{d_j-1\leq k \leq d_1+d_j-2}(2d_j+1) 
 a_{k-d_1'-d_j+1} [\tau_{k} \prod_{i \neq 1, j} \tau_{d_i}]_{g,n-1}
 \\&\leq& 
  \sum_{d_1+d_j-1\leq k \leq d_0+d_1+d_j-1}(2d_j+1) 
c_0 2^{-2(k-d_1-d_j+1)} V_{g,n-1}
\\&& + \sum_{d_j-1\leq k \leq d_1+d_j-2}(2d_j+1) 
 V_{g,n-1} 
  \\&\leq& C (d_1 +2) (2d_j+1)  V_{g,n-1}, 
\end{eqnarray*} 
 for some universal constants $C$.  

 Using \cite[inequality (3.9)]{Mirzakhani:Growth} to uniformly bound the ratio $\frac{V_{g,n-1}}{V_{g,n}}$ of volumes and summing over $j$,  we conclude 
$$\frac{\cA_\mathbf{d'} - \cA_\mathbf{d}}{V_{g,n}} \leq   \frac{ C \left( (n-1) (d_1 + 2) + 4 |\mathbf{d}| +  \sum_{j = 2}^n d_1 d_j  \right) }{2g-3+n} $$ 
for a different universal constant $C>0.$
 \begin{comment}
 Using \cite[inequality (3.9)]{Mirzakhani:Growth} to uniformly bound the ratio of volumes and summing over $j$, we conclude 
$$\cA_\mathbf{d'} - \cA_\mathbf{d} \leq   \frac{ C |\mathbf{d}|^2V_{g,n}}{2g-3+n}$$
for a different universal constant $C>0$, where $\cA_\mathbf{d}=\sum_{j=2}^n \cA_\mathbf{d}^j$.
\bigskip
\bigskip
Using \cite[inequality (3.9)]{Mirzakhani:Growth} to uniformly bound the ratio of volumes and summing over $j$, we conclude 
$$\cA_\mathbf{d'} - \cA_\mathbf{d} \leq   \frac{ C \left( (n-1) (d_1 + 2) + 4 |\mathbf{d}| +  \sum_{j = 2}^n d_1 d_j  \right)  V_{g,n}}{2g-3+n}$$
for a different universal constant $C>0.$
\end{comment}
\end{proof}

\begin{cor}[Total $\cA$-term bound] \label{totalAbound}
Let $n = o(\sqrt{g}).$  There is a constant $C' > 0$ for which
\begin{equation*}  
\frac{\cA_\mathbf{0}-\cA_\mathbf{d}}{V_{g,n}} \leq \frac{ C' \left( n |\mathbf{d}| + |\mathbf{d}|^2 \right) }{2g-3+n}.
 \end{equation*}
\end{cor}

Here we use the notation $\mathbf{0}=[0,\ldots,0]$.

\begin{proof}
Let $\mathbf{d}_i$ be gotten from $\mathbf{d}$ by setting the first $i$ non-zero coordinates of $\mathbf{d}$ to 0.  Let $k$ be the number of non-zero coordinates of $\mathbf{d}.$  Using that $k \leq |\mathbf{d}|,$ it follows from Lemma \ref{incrementalAbound} that
\begin{align*}
\frac{\cA_\mathbf{0}-\cA_{\mathbf{d}}}{V_{g,n}} &= \sum_{i = 0}^{k - 1} \frac{\cA_{\mathbf{d}_{k - i}} - \cA_{\mathbf{d}_{k - (i+1)}}}{V_{g,n}} \\
&\leq \frac{ C \left( (n-1) |\mathbf{d}| + (n-1) \cdot 2 \cdot k + 4 |\mathbf{d}| \cdot k + |\mathbf{d}|^2 \right) }{2g-3+n} \\
&\leq \frac{ C' \left( n |\mathbf{d}| + |\mathbf{d}|^2 \right) }{2g-3+n}
\end{align*}
for $C' = 5C.$
\end{proof}

We now turn to the $\cB$ term, continuing to use $\mathbf{d'} = (0, d_2, \ldots, d_n)$. 

\begin{lem}[Incremental $\mathcal{B}$-term bound] \label{incrementalBbound}
Let $n = o(\sqrt{g}).$  There is an absolute constant $C'' > 0$ for which
\begin{equation*}  
\frac{\cB_\mathbf{d'} - \cB_\mathbf{d}}{V_{g,n}}  \leq \frac{ C'' d_1^2}{2g - 3 + n}.
 \end{equation*}
\end{lem}

\begin{proof}
We compute 
\begin{align*}
&(\cB_\mathbf{d'} - \cB_\mathbf{d})/16 \\
&= \sum_{d_1 - 2 \leq k_1 + k_2 \leq d_0+ d_1 - 2}(a_{k_1 + k_2 + 2} - a_{k_1 + k_2 - d_1 + 2}) [\tau_{k_1} \tau_{k_2} \prod_{i \neq 1} \tau_{d_i}]_{g-1,n+1} \\
&+ \sum_{0 \leq k_1 + k_2 < d_1 - 2} a_{k_1 + k_2 + 2}  [\tau_{k_1} \tau_{k_2} \prod_{i \neq 1} \tau_{d_i}]_{g-1,n+1}  \\
&\leq \sum_{d_1 - 2 \leq k_1 + k_2 \leq d_0+ d_1 - 2} c_0 2^{-2(k_1 + k_2 - d_1 + 2)} V_{g-1,n+1} \\
&+  \sum_{0 \leq k_1 + k_2 < d_1 - 2} 1 \cdot V_{g-1,n+1}\\
% &\leq C d_1 \cdot V_{g-1,n+1}  \\
% &+ \frac{1}{2} d_1^2 \cdot V_{g-1,n+1} \\
&\leq C' d_1^2 \cdot V_{g-1,n+1}.
\end{align*}
From here we conclude as in the proof of Lemma \ref{incrementalAbound}.
% Thus,\todo{June 26 Mike: Maybe we can just omit starting from here, since we can then see the slight difference but then not repeat ourselves with volume bounds. A: I like this idea!}  \textcolor{blue}{For consistency with the $\cA$-bound, should we just quote the same volume ratio bounds instead of writing out the below inequalities?}
% \begin{align*}
% \frac{\cB_{\mathbf{d'}} - \cB_{\mathbf{d}}}{V_{g,n}} & \leq 16C' d_1^2 \cdot \frac{V_{g-1,n+1}}{V_{g,n}} \\
% &\leq 16C' d_1^2 \cdot \frac{\frac{1}{\max \{1,\sqrt{g-1}\}} (2(g-1) -3 + n + 1)! (4\pi^2)^{2(g-1)-3+ n+1}}{ \frac{1}{\max \{1,\sqrt{g}\}} (2g-3+n)! (4\pi^2)^{2g-3+n} } \\
% &\leq C'' d_1^2 \cdot \frac{1}{2g-3 + n}.
% \end{align*}
\end{proof}

\begin{cor}[Total $\cB$-term bound] \label{totalBbound}
Let $n = o(\sqrt{g}).$  There is a constant $C'' > 0$ for which
\begin{equation*}  
\frac{\cB_\mathbf{0}-\cB_\mathbf{d}}{V_{g,n}} \leq \frac{C'' |\mathbf{d}|^2}{2g - 3 + n}.
 \end{equation*}
\end{cor}

For the $\cC$ term, we will use the following. 

\begin{lem} \label{binomialinequality}
Suppose $N=o(\sqrt{G})$. Then 
$$\sum_{s=0}^{\frac{G}2} \sum_{k=0}^{\min\{s+2, N\}} {N \choose k} {G \choose s}^{-1}$$
is $O((N+1)^2)$. 
\end{lem}

% Formally this means that for every function $f\in o(\sqrt{G})$ there is a constant $C$ such that the sum is at most $C (N+1)^2$ whenever $N\leq f(G)$. 

\begin{proof}
We first bound the contribution of $s\geq N$ as follows. 
$$\sum_{s=N}^{\frac{G}2} \sum_{k=0}^{\min\{s+2, N\}} {N \choose k} {G \choose s}^{-1} \leq 2^N \sum_{s=N}^{\frac{G}2} {G \choose s}^{-1}.$$
Using that the sum is at most the number of terms times the largest term, together with the inequality ${G \choose N}\geq (G/N)^N$, we get that this is at most 
$$2^N G {G \choose N}^{-1} \leq 2^N G \left(\frac{N}{G}\right)^{N} \leq  G \left(\frac{2N}{G}\right)^{N},$$
which goes to $0$ as $G\to\infty$ as long as $N\geq 2$. 

For the cases when $N< 2$, to see that the contribution from $s\geq N$ is bounded it suffices to note that $\sum_{s=0}^{\frac{G}2} {G \choose s}^{-1}$ is bounded. % by first separating off the first term and then using the ``number of terms times largest term" bound.

We now bound the contribution of smaller $s$ as follows. 
\begin{eqnarray*}
\sum_{s=0}^{N-1} \sum_{k=0}^{\min\{s+2, N\}} {N \choose k} {G \choose s}^{-1} 
&\leq& \sum_{s=0}^{N-1} (N+1)^{s+2} {G \choose s}^{-1} 
\\&\leq& 
\sum_{s=0}^{N-1} (N+1)^{s+2} \left(\frac{s}{G}\right)^s
\\&\leq& 
(N+1)^2\sum_{s=0}^{N-1}  \left(\frac{(N+1) s}{G}\right)^s
\end{eqnarray*}
This is $O((N+1)^2)$ since $N$ is $o(\sqrt{G})$ and $s\leq N-1$. 
\end{proof}

\begin{lem}[Incremental $\cC$-term bound] \label{incrementalCbound}
Let $n = o(\sqrt{g}).$  There is a constant $C''' > 0$ such that
\begin{equation*}  
\frac{\cC_\mathbf{d'} - \cC_\mathbf{d}}{V_{g,n}} \leq \frac{C''' \cdot d_1^2 \cdot n^2}{(2g - 3 + n)^2}.
 \end{equation*}
\end{lem}

\begin{proof}
We can and will assume $d_1>0$, since if $d_1=0$ then $\cC_\mathbf{d'} = \cC_\mathbf{d}$. Define 
$$
    S_{g,n} = \sum_{\substack{g_1+g_2=g\\ I \sqcup J = \{2, \ldots, n\}   }} 
   V_{g_1, |I|+1} \cdot V_{g_2, |J| + 1} \\
  = \sum_{\substack{g_1+g_2=g \\ 0 \leq k \leq n-1 }}
 \binom{n-1}{k} V_{g_1, k + 1} \cdot V_{g_2, n-1 - k + 1}.
$$
We compute
\begin{eqnarray*} 
(\cC_\mathbf{d'} - \cC_\mathbf{d})/16 
% &=&  \sum_{\substack{g_1+g_2=g\\ I \sqcup J = \{2, \ldots, n\} \\ d_1-2\leq k_1+k_2\leq d_0+d_1 -2 }} 
%  (a_{k_1+k_2+2} - a_{k_1+k_2-d_1+2}) \cdot  [\tau_{k_1} \prod_{i\in I} \tau_{d_i}]_{g_1, |I|+1}  
%  [\tau_{k_2} \prod_{i\in J} \tau_{d_i}]_{g_2, |J|+1} \nonumber \\
%  &&+  \sum_{\substack{g_1+g_2=g\\ I \sqcup J = \{2, \ldots, n\} \\ -2\leq k_1+k_2 < d_1 -2 }} 
%  a_{k_1+k_2+2} \cdot [\tau_{k_1} \prod_{i\in I} \tau_{d_i}]_{g_1, |I|+1}  
%  [\tau_{k_2} \prod_{i\in J} \tau_{d_i}]_{g_2, |J|+1} \nonumber \\
 % &\leq&  \sum_{\substack{g_1+g_2=g\\ I \sqcup J = \{2, \ldots, n\} \\ d_1-2\leq k_1+k_2\leq d_0+d_1 -2 }} 
 % (a_{k_1+k_2+2} - a_{k_1+k_2-d_1+2}) \cdot V_{g_1, |I|+1} \cdot V_{g_2, |J| + 1} \nonumber \\
 % &&+  \sum_{\substack{g_1+g_2=g\\ I \sqcup J = \{2, \ldots, n\} \\ -2\leq k_1+k_2 < d_1 -2 }} 
% a_{k_1+k_2+2} \cdot V_{g_1, |I|+1} \cdot V_{g_2, |J| + 1} \nonumber \\
 % &\leq&  \left( \sum_{d_1-2\leq k_1+k_2\leq d_0+d_1 -2} (a_{k_1+k_2+2} - a_{k_1+k_2-d_1+2}) \right) \cdot \sum_{\substack{g_1+g_2=g\\ I \sqcup J = \{2, \ldots, n\} }} 
 % V_{g_1, |I|+1} \cdot V_{g_2, |J| + 1} \nonumber \\
 % &&+  \left( \sum_{-2\leq k_1+k_2 < d_1 -2} a_{k_1+k_2+2} \right) \cdot \sum_{\substack{g_1+g_2=g\\ I \sqcup J = \{2, \ldots, n\}   }} 
 %  \cdot V_{g_1, |I|+1} \cdot V_{g_2, |J| + 1} \nonumber \\
  &\leq&  \left( \sum_{d_1-2\leq k_1+k_2\leq d_0+d_1 -2} (a_{k_1+k_2+2} - a_{k_1+k_2-d_1+2}) \right) \cdot S_{g,n} \label{overlap} \\
 &&+  \left( \sum_{-2\leq k_1+k_2 < d_1 -2} a_{k_1+k_2+2} \right) \cdot S_{g,n}. \label{nonoverlap}
\end{eqnarray*}
Since $0 \leq a_{k_1+k_2+2} - a_{k_1+k_2-d_1+2} \leq \frac{c_0}{2^{2\cdot(k_1 + k_2 - d_1 + 2)}}$ for some  constant $c_0$,
\begin{align*}
 \sum_{d_1-2\leq k_1+k_2\leq d_0+d_1 -2} (a_{k_1+k_2+2} - a_{k_1+k_2-d_1+2}) &\leq \sum_{d_1-2\leq k_1+k_2} \frac{c_0}{2^{2(k_1 + k_2 - d_1 + 2)}}  
 %\\
% &= \sum_{j = 0}^\infty (j+1 + d_1) \frac{c_0}{2^{2j}} \\
% &= c_0 \cdot \left( \frac{16}{9} + \frac{4}{3} \cdot d_1 \right).
\end{align*}
is at most a constant times $d_1$. Since $a_w \leq 1$ for all $w,$ 
$$ \sum_{-2\leq k_1+k_2 < d_1 -2} a_{k_1+k_2+2}\leq \frac{1}{2}d_1(d_1 + 1),$$ and we can conclude that 
$\frac{\cC_\mathbf{d'} - \cC_\mathbf{d}}{S_{g,n}} $
is at most a constant times $d_1^2$.

It remains to bound $\frac{S_{g,n}}{V_{g,n}},$ towards which we begin by noting that
\begin{align*}
S_{g,n} %&= \sum_{\substack{g_1+g_2=g \\ 0 \leq k \leq n-1 }}
 %\binom{n-1}{k} V_{g_1, k + 1} \cdot V_{g_2, n-1 - k + 1} \\
 &= \sum_{R_0}\binom{n-1}{k} V_{g_1, k + 1} \cdot V_{g_2, n-1 - k + 1},
\end{align*}
where $R_0$ denotes the lattice points $(g_1,k)$ in the rectangular region 
$$R := \{(g_1, k):  0 \leq g_1 \leq g \text{ and } 0\leq k \leq n-1 \}.$$  
The expression being summed is invariant under $$(g_1, k) \leftrightarrow (g-g_1, n-1-k),$$ corresponding to point reflection $\iota$ about the center $(\frac{g}{2}, \frac{n-1}{2})$ of $R.$  Thus,
\begin{equation*}
S_{g,n} \leq 2 \cdot \sum_{R_0'} \binom{n-1}{k} V_{g_1, k + 1} \cdot V_{g_2, n-1 - k + 1},
\end{equation*}
where $R_0'$ is any ``half" of $R_0$, i.e. any subset $S \subset R_0$ for which $S \cup \iota S = R_0.$

Applying \eqref{E:asym} and \eqref{E:upper}, we get
\begin{align*}
 \frac{V_{g_1, k + 1} \cdot V_{g_2, n-1 - k + 1}}{V_{g,n}} 
% &\leq C \cdot \frac{\frac{1}{\max \{1,\sqrt{g_1}\}} (2g_1 -3 + k + 1)! (4\pi^2)^{2g_1-3+ k + 1} \times  \frac{1}{\max \{1,\sqrt{g_2}\}} (2g_2-3+n-k)! (4\pi^2)^{2g_2-3+n-k}}{ \frac{1}{\max \{1,\sqrt{g}\}} (2g-3+n)! (4\pi^2)^{2g-3+n} } \\
&\leq C' \cdot \frac{ (2g_1 -3 + k + 1)!  \cdot  (2g_2-3+n-k)! }{  (2g-3+n)! } \\
&= C' \cdot \frac{1}{(2g - 3 + n)(2g - 4 + n)} \cdot \binom{2g - 5 + n}{2g_1 - 3 + k + 1}^{-1}
\end{align*}
for some $C'>0$. (See Lemma \ref{L:ProductBound} for a similar  estimate.)
 %V_{g,n} \leq \frac{C_0}{\max \{1,\sqrt{g}\}} (2g-3+n)! (4\pi^2)^{2g-3+n}
Thus,
\begin{equation*}%\label{binomialsum}
\frac{S_{g,n}}{V_{g,n}} \leq C'' \cdot \frac{1}{(2g - 3 + n)(2g - 4 + n)} \cdot 2 \cdot \sum_{R_0'} 
 \binom{n-1}{k} \binom{2g - 5 + n}{2g_1 - 2 + k}^{-1},
 \end{equation*}
for any half $R_0'$ of $R_0.$
Let $s = 2g_1 - 2 + k.$  We choose the half $R_0'$ consisting of all points of $R_0$ satisfying 
$s \leq 2 \cdot \frac{g}{2} - 2 + \frac{n-1}{2} = \frac{1}{2}(2g - 5 + n).$  Parametrizing $R_0'$ in $(s,k)$ coordinates, this amounts to all lattice points $(s,k)$ satisfying
\begin{itemize}
\item 
$s \in [0, \frac{1}{2} \cdot ( 2g-5+n )],$

\item 
$0 \leq k \leq n-1$ and $k \leq s + 2$ (corresponding to $g_1 \geq 0$), i.e. $k \in [0, \min\{n-1, s+2\}].$

\item 
$k \equiv s \mod 2.$
\end{itemize}

 It follows that
\begin{equation*}\label{binomialsumsub} 
\frac{S_{g,n}}{V_{g,n}} \leq   \frac{2C''}{(2g - 3 + n)(2g - 4 + n)}   \sum_{s = 0}^{\frac{1}{2}(2g - 5 + n)} \sum_{k = 0}^{\min\{n-1,s+2\}}
 \binom{n-1}{k} \binom{2g - 5 + n}{s}^{-1},
 \end{equation*}
and hence Lemma \ref{binomialinequality} gives
 \begin{equation*}
\frac{S_{g,n}}{V_{g,n}} \leq C''' \cdot \frac{1}{(2g - 3 + n)(2g - 4 + n)}  \cdot n^2,
 \end{equation*}
concluding the proof.
% Combining with \eqref{nonoverlap} and \eqref{overlap}, it follows that
% \begin{equation*}
%      \frac{\cC_\mathbf{d'} - \cC_\mathbf{d}}{V_{g,n}} \leq D \cdot (E \cdot d_1^2 + F) \cdot n^2 \frac{1}{(2g-3 + n)(2g-4+n)}
% \end{equation*}
% for absolute constants $D,E,F > 0.$ 
\end{proof}

\begin{cor}[Total $\cC$-term bound] \label{totalCbound}
Let $n = o(\sqrt{g}).$  There is a constant $C''' > 0$ for which 
\begin{equation*}  
\frac{\cC_\mathbf{0}-\cC_\mathbf{d}}{V_{g,n}} \leq  \frac{C''' \cdot |\mathbf{d}|^2 \cdot n^2}{(2g - 3 + n)^2}.
 \end{equation*}
\end{cor}

% \begin{proof}
%     Let $\mathbf{d}_i$ be gotten from $\mathbf{d}$ by setting the first $i$ non-zero coordinates of $\mathbf{d}$ to 0.  Let $k$ be the number of non-zero coordinates of $\mathbf{d}.$  Noting that  
% \begin{equation*}
% \frac{\cC_\mathbf{0}-\cC_{\mathbf{d}}}{V_{g,n}} = \sum_{i = 0}^{k - 1} \frac{\cC_{\mathbf{d}_{k - i}} - \cC_{\mathbf{d}_{k - (i+1)}}}{V_{g,n}}, 
% \end{equation*}
% the result follows immediately upon summing the inequality from Lemma \ref{incrementalCbound} $k$-times.
% \end{proof}

\begin{proof}[Proof of Theorem \ref{T:Coeffs}.]
Given Corollary \ref{C:UpperB},  it suffices to prove the upper bound. Note that 
\begin{eqnarray*}
1- \frac{[\tau_{d_1} \cdots \tau_{d_n}]_{g,n}}{V_{g,n}}
&=& \frac{\cA_0-\cA_\mathbf{d}}{V_{g,n}} + \frac{\cB_0-\cB_\mathbf{d}}{V_{g,n}} + \frac{\cC_0-\cC_\mathbf{d}}{V_{g,n}}
\\&\leq&   \frac{  C' n |\mathbf{d}| + |\mathbf{d}|^2  }{2g-3+n}
+ \frac{C'' |\mathbf{d}|^2}{2g - 3 + n} 
+ \frac{C''' \cdot |\mathbf{d}|^2 \cdot n^2}{(2g - 3 + n)^2}
\end{eqnarray*}
where the inequality follows from Corollaries \ref{totalAbound}, \ref{totalBbound}, and \ref{totalCbound}. Keeping in mind that $n\geq 1$ and $n=o(\sqrt{g})$, the result follows. 
\end{proof}

\section{Volume bounds}\label{A:Vols}

\subsection{The sinh estimate and exponential upper bound.} The following is a version of \cite[Proposition 3.1]{MirzakhaniPetri:Lengths}.

\begin{lem}[Mirzakhani-Petri]\label{L:sinh}
In general,
$$ \frac{V_{g,n}(2L_1, \ldots, 2L_n)}{V_{g,n}} \leq \prod_{i=1}^n \frac{\sinh(L_i)}{L_i} \leq \exp\left(\sum L_i \right).$$
If we assume $n=g^{o(1)}$ and $\sum L_i \leq g^{o(1)}$ then 
$$\left( 1-g^{-1+o(1)} \right) \prod_{i=1}^n \frac{\sinh(L_i)}{L_i} \leq \frac{V_{g,n}(2L_1, \ldots, 2L_n)}{V_{g,n}}.$$
\end{lem}

\begin{proof}
The upper bound follows immediately from the inequality 
$$\frac{[\tau_{d_1} \cdots \tau_{d_n}]_{g,n}}{V_{g,n}} \leq 1.$$

For the lower bound, note:
\begin{enumerate}
\item
$\prod_{i=1}^n \frac{\sinh(L_i)}{L_i}$ is extremely close to $\sum_{|\mathbf{d}| \leq 3g-3 + n} \frac{L_1^{2d_1}}{(2d_1+1)!} \cdots \frac{L_n^{2d_n}}{(2d_n+1)!}$; the latter is a high truncation of the power series for the former because $n=g^{o(1)}$.

\item
By Theorem \ref{T:Coeffs}, the difference
$\sum_{|\mathbf{d}| \leq 3g-3 + n} \frac{L_1^{2d_1}}{(2d_1+1)!} \cdots \frac{L_n^{2d_n}}{(2d_n+1)!} - \frac{V_{g,n}(2 \mathbf{L})}{V_{g,n}}$
%$\prod_{i=1}^n \frac{\sinh(L_i)}{L_i} - \frac{V_{g,n}(2 \mathbf{L})}{V_{g,n}}$
is bounded above by
\begin{align}
&\frac{Cn}{2g-3 + n} \cdot \sum_{\mathbf{d}} |\mathbf{d}|^2 \frac{L_1^{2d_1}}{(2d_1+1)!} \cdots \frac{L_n^{2d_n}}{(2d_n+1)!} \nonumber \\
= &\frac{Cn}{2g-3 + n} \cdot \sum_{i,j} \sum_{\mathbf{d}} d_i d_j \cdot \frac{L_1^{2d_1}}{(2d_1+1)!} \cdots \frac{L_n^{2d_n}}{(2d_n+1)!}. \label{ijsummand}
\end{align}

The $i,j$ summand in \eqref{ijsummand} splits as the product

\begin{equation} \label{splitproduct}
\begin{cases}
P(L_k) \prod_{\ell \neq k} \frac{\sinh(L_\ell)}{L_\ell} &\text{ if } i = j = k \\
Q(L_i) Q(L_j) \prod_{\ell \neq i,j} \frac{\sinh(L_{\ell})}{L_{\ell}} &\text{ if } i \neq j,
\end{cases}
\end{equation}
where $P(x) = \sum_{d \geq 0} d^2 \cdot \frac{x^{2d}}{(2d+1)!}$ and $Q(x) = \sum_{d \geq 0} d \cdot \frac{x^{2d}}{(2d+1)!}.$  The product in the first case of
\eqref{splitproduct} is uniformly comparable to $L_k^2 \prod_{i = 1}^n \frac{\sinh(L_i)}{L_i},$ and the product in the second case of \eqref{splitproduct} is uniformly comparable
to $L_i L_j \cdot \prod_{i = 1}^n \frac{\sinh(L_i)}{L_i}.$

\end{enumerate}

The lower bound follows.
\end{proof}

\subsection{Volumes of boundary strata.} Our next results, combined, are a variant of \cite[Lemma 3.2]{MirzakhaniPetri:Lengths}, which in turn extends \cite[Lemma 3.3]{Mirzakhani:Growth}. 

\begin{lem}\label{L:ProductBound}
There exists a  constant $C_1$  such that 
$$\frac1{V_g} \prod_{i=1}^q V_{g_i, n_i}
\leq           \left( \frac{C_1}{g}\right)^{q+q'-2},
$$
provided
\begin{enumerate}
\item $q'$ is the number of $i$ with $(g_i, n_i) \notin \{(0,3),(1,1)\},$
\item $q\geq 2$, $g\geq 2$,  $k\leq g/4$, $q+q'>2$,
\item $\sum_{i=1}^q n_i = 2k$,
\item $\sum_{i=1}^q g_i = g+q-k-1$, and 
\item $2g_i-3+n_i \geq 0$ and $n_i\geq 1$ for all $i=1, \ldots, q$. 
\end{enumerate}
If, additionally, $q=2$ and $n_1=n_2=k$ and both $2g_i + k - 3  \geq b$, we have 
$$\frac{V_{g_1, k}  V_{g_2,k}}{V_g} \leq C_1 \frac{b^b}{g^{b+1}}.$$
\end{lem}

\begin{proof}
Our assumptions imply 
$$(2g-3) - \sum (2g_i - 3 +n_i) = q-1.$$
 The asymptotic \eqref{E:asym} implies that there is a $D_0>0$ such that for all $g\geq 2$, 
\begin{equation}
\label{E:lower}
V_{g} \geq \frac{D_0}{\sqrt{g}} (2g-3)! (4 \pi^2)^{2g-3}.    
\end{equation}
 Thus  \eqref{E:upper} implies that 
$$
\frac1{V_g} \prod V_{g_i, n_i}
\leq
(C_0/D_0)\cdot (C_0 / 4\pi^2) ^{q-1} \cdot \frac{\sqrt{g}}{ \prod \max \{1, \sqrt{g_i} \}} \cdot \frac{\prod (2g_i-3+n_i)!}{(2g-3)!}.
$$

\begin{sublem}
Given our assumptions, 
$$\frac{\sqrt{g}}{ \prod \max \{1,\sqrt{g_i} \}}$$
is bounded above by a constant. 
\end{sublem}
\begin{proof}
Since $q\leq k+1$, we  have that $q\leq g/4+1$. We also have that $\sum g_i \geq g/2$. 

In general, if $\sum_{i=1}^q x_i = x$ and all $x_i\geq 1$, then 
$$ \prod x_i \geq x-(q-1).$$
The sublemma follows. 
\end{proof}

The first claim now follows from the  inequality 
$$\prod_{i=1}^{q'} \ell_i! \leq (\ell - q'+1)!$$
for positive integers $\ell_i$ that sum to $\ell$. 

 To   get the second claim, note that under the additional assumptions the arguments above imply an upper bound of a constant times 
$$ \frac{1}{g} {{2g-4} \choose 2g_1 +k -3}^{-1}.$$

Thus the bound 
$$ {\ell \choose b} \geq \frac{\ell^b}{b^b}$$
gives the second claim.  
\end{proof}

The second part of the previous  result implies \cite[Lemma 3.3]{Mirzakhani:Growth}, which we restate for convenience. 

\begin{cor}[Mirzakhani]\label{C:EasierSum}
For all non-negative integers $b$ there is a constant $C=C(b)$ such 
$$\sum \frac{V_{g_1,k} V_{g_2,k}}{V_g} \leq C g^{-b-1},$$
where the sum is over triples $(g_1, g_2, k)$ with both $2g_i+ k-3 \geq b$ that correspond to pinching a multi-curve with $k$ components whose complement has two components, of genus $g_1$ and $g_2$.  
\end{cor}

\begin{proof}
There are $O(g^2)$ terms, so summing just the terms with $2g_i+ k-3 \geq b+2$ gives the desired bound when summing over all these terms, by the result above. 

Given that $b$ is constant, the number of triples where one of the $2g_i+ k-3 $ is $b+1$ or $b$ is $O(1)$, so again we get the result summing over these. 
\end{proof}

For more complicated applications we also need the following.

\begin{lem}\label{L:CountStrata}
The number of strata in the Deligne-Mumford compactification where $k$ curves have been pinched and $q$ components produced, $q'$ of which aren't spheres with three marked points or tori with one marked point, is at most  
$$2^{k+q^2} g^{q'-1}.$$
\end{lem}

This  bound is not sharp, but is sufficient for our purposes. 

\begin{proof}
First pick how many of the $q-q'$ small components have genus 1 and how many have genus 0. There are at most $q$ possibilities. The remaining components each have genus at most $g-1$, and the sum of the $g_i$ is known, so there are at most $g^{q'-1}$ many ways to pick the genera of the remaining components. 

%The genera of the components add up to $g+q-k-1 \leq g$, so there are at most 
%$$ {{g+q'-1} \choose {q'-1}}       \leq g^{q'-1} $$
%ways of picking a tuple $(g_1, \ldots, g_{q'})$ of the remaining components with the correct sum.  

There are  $q(q+1)/2$ ways to add a node, since one simply needs to pick the two components (possibly the same) the node will be on. Thus the number of ways to add the nodes is bounded by the number of $q(q+1)/2$-tuples of non-negative integers that add up to $k$, which is 
$${k +  q(q+1)/2 - 1 \choose k} \leq 2^{k +  q(q+1)/2 - 1}.$$
Thus the number of strata is bounded by this quantity times $q g^{q'-1}.$
\end{proof}

\subsection{Separating curves.} We now apply the results above, following \cite{Mirzakhani:Growth} and  \cite{MirzakhaniPetri:Lengths}.

\begin{cor}\label{C:SepVol}
For any integer $a\geq 0$, the probability that a surface in $\cM_g$ has a multi-geodesic of length at most $L$ whose complement has two components, each of area at least $2\pi a$ is 
$$  O(e^{2L} g^{-a}) .$$

For fixed $k$, the average number of such multi-geodesics with exactly $k$ components is 
$$  O(e^{L} L^{2k} g^{-a}) .$$

The average number of  such  multi-geodesics % of length at most $L$ 
bounding a subsurface of area exactly $2\pi a$ is 
$$ O(e^{L/2} (L^2+L^p) g^{-a})$$
for some $p\geq 2$.  
\end{cor}

\begin{proof}
Denote the genera of the two components by $g_1$ and $g_2$. 
The area of a subsurface is $2\pi$ times its Euler characteristic, so area at least $2\pi a$ is equivalent to $2g_i + k -2 \geq a$, where the subsurface has genus $g_i$ and $k$ boundary components. 

 The probability in question is bounded by 

$$\frac1{V_g} \sum_{2g_i+k-2\geq a} \int_{L_1+ \cdots +L_k \leq L} L_1 \cdots L_k V_{g_1, k}(L_1, \ldots, L_k) V_{g_2, k}(L_1, \ldots, L_k) \; dL_1 \cdots dL_k.$$

% \textcolor{red}{This is a bound and not an equality because you're not putting any lower bound on $g_2$?}
% \textcolor{blue}{This is for ordered collections of curves, i.e., I didn't want to think about the constants in the integration formula. But, if we assume every component of the multi-curve bounds both sides of the complement (no redundancy) probably we could save a $k!$ here.}

Using the exponential upper bound, this is at most 

$$e^L \sum_{2g_1+k-2\geq a} \frac{V_{g_1, k} V_{g_2, k}}{V_g} \int_{L_1+ \cdots +L_k \leq L} L_1 \cdots L_k \; dL_1 \cdots dL_k.$$

Using the value of this integral given in \cite[Proof of Lemma 4.9]{Mirzakhani:Growth},  this is equal to 

$$e^L \sum_{2g_i+k-2\geq a} \frac{V_{g_1, k} V_{g_2, k}}{V_g}\frac{ L^{2k}}{(2k)!}.$$

Now, since $e^L$ is greater than any term in its Taylor series, 
%\textcolor{red}{Might be pretty wasteful, no?  Where exactly does this enter into our other estimates?} 
this is bounded by 

$$e^{2L} \sum_{2g_i+k-2\geq a} \frac{V_{g_1, k} V_{g_2, k}}{V_g}.$$
So Corollary \ref{C:EasierSum} gives the first claim.

The other claims are similar: For example, when $2g_1+k=a$ is fixed,  we can assume $g_1$ and $k$ are fixed. In comparison to the first claim, now $V_{g_1, k}(L_1, \ldots, L_k)$  and  $\frac{ L^{2k}}{(2k)!}$ bounded by fixed polynomials in $L$. 
\end{proof}

\section{Local Weyl law} \label{localweyllawappendix}  
%and bounds for the spectral side of the trace formula

%\subsection{Local Weyl law}

We start with the following local Weyl law. 

\begin{prop} \label{localweyllaw}
Fix $\mu_0 > 0.$  For every closed hyperbolic surface $X,$ the number of spectral parameters $r_n(X) = \sqrt{\frac{1}{4} - \lambda_n(X)}$ satisfying $|r_n(X)| \in [-1 + t, 1 + t]$ is bounded above by 
%$$(g-1) \cdot \left( A |t| + B \right) + C \cdot N_X \cdot \left( \log \left(  \frac{\mu_0 + |t|}{\mathrm{sys}(X)} \right) + 1 \right) + D \cdot E_X,$$

%$$(g-1)  \left( A |t| + B \right) + C  N_X  \logp\left(  \frac{\mu_0 + |t|}{\mathrm{sys}(X)} \right) + D  E_X,$$

$$(g-1)  \left( A |t| + B \right) + C  N_X  \logp\left(  \frac{\mu_0}{\mathrm{sys}(X)} \right) + D  E_X,$$ 
where 
\begin{itemize}
\item
$N_X$ is the number of primitive closed geodesics of length at most $\mu_0,$ 

\item
$E_X$ is the number of exceptional parameters  $r_n(X) \in (0\cdot i,1/2 \cdot i].$

\item
$\mathrm{sys}(X)$ is the systole of $X,$ and 

\item $\logp(x) = \max \{0, \log(x) \} + 1$. 

\item
$A,B,C,D$ are constants depending only on $\mu_0.$  
\end{itemize} 
\end{prop}

\begin{proof}
We prove this using the trace formula, which is recalled in Theorem \ref{T:TF} together with our notation and normalizations.  Let $f$ be an even, smooth real-valued function supported on $[-\mu_0,\mu_0]$ for which $\widehat{f}$ is non-negative and $\widehat{f}(r) > 0$ for $r \in [-1,1].$  Define $$m = \min \{ \widehat{f}(r)/2: r \in [-1,1] \},$$  
and $f_t(x) = f(x) \cdot \cos(tx)$, so that 
$$\widehat{f_t}(r) = \frac{1}{2} \left( \widehat{f}(r - t) + \widehat{f}(r + t) \right).$$
Note%\footnote{Recall our earlier notation: for smooth, even, compactly supported test functions $h,$ $$F_h(x) := x \cdot \sum_{k = 1}^\infty \frac{h(kx)}{2 \sinh(kx/2)}.$$}
\begin{eqnarray*}
&& m \cdot \# \{ r_n(X) \in [-1 + t, 1 + t] \} + \sum_{r_n = ib_n} \int_{\mathbb{R}} f(x) e^{b_n x} \cos(tx) dx \\
&\leq& \sum_n \widehat{f_t}(r_n(X)) \\
&=& %\sum_{\gamma \in \mathcal{P}(X)} F_{f_t}( \ell(\gamma) ) 
F_{f_t, \mathrm{all}}(X)
+ (g-1)   \int_{-\infty}^\infty \widehat{f_t}(r) r \tanh(\pi r) dr. 
\end{eqnarray*}

The function $\frac{x}{\sinh(x/2)}$ is strictly positive and at most 2 for $x \geq 0.$  So for primitive geodesics $\gamma,$ the fact that $f_t$ is supported on 
 $[-\mu_0, \mu_0]$
 implies
\begin{eqnarray*}
\left|F_{f_t}(\ell(\gamma)) \right| &\leq& \|f_t\|_{\infty}  \sum_{k \leq \frac{\mu_0}{\ell(\gamma)} } \frac{1}{k}  
\\&\leq& \|f\|_{\infty}  \logp\left( \frac{\mu_0}{\mathrm{sys}(X)} \right).
\end{eqnarray*}

% \color{red}
% Corrected version: implies
% \begin{eqnarray*}
% \left|F_{f_t}(\ell(\gamma)) \right| &\leq& \|f_t\|_{\infty}  \sum_{k \leq \frac{\mu_0}{\ell(\gamma)} } \frac{1}{k}  
% \\&\leq& \|f\|_{\infty}  \logp\left( \frac{\mu_0}{\mathrm{sys}(X)} \right).
% \end{eqnarray*}
% \color{black}

% \begin{align*}
% \left|F_{f_t}(\ell(\gamma)) \right| &\leq \|f_t\|_{\infty} \cdot \sum_{k \leq \frac{\mu_0 + |t|}{\ell(\gamma)} } \frac{1}{k}  \leq \|f\|_{\infty} \cdot \left( \log \left( \frac{\mu_0 + |t|}{\mathrm{sys}(X)} \right) + 1 \right).
% \end{align*}

So the first summand $F_{f_t, \mathrm{all}}(X)$ above is bounded above by $$C_0 N_X \logp\left( \frac{\mu_0}{\mathrm{sys}(X)} \right),$$ where $C_0 = \|f\|_{\infty}.$

% \color{red}
% Corrected version: So the first summand $F_{f_t, \mathrm{all}}(X)$ above is bounded above by $$C_0 N_X \logp\left( \frac{\mu_0}{\mathrm{sys}(X)} \right),$$ where $C_0 = \|f\|_{\infty}.$
% \color{black}

For the second summand, changing variables shows immediately that the integral is bounded above by $A_0 |t| + B_0$ for some absolute constants $A_0, B_0.$ 
Note also that
$$ \sum_{r_n = ib_n} \int_{\mathbb{R}} f(x) e^{b_n x} \cos(tx) dx  \geq - D_0  E_X, \text{ where } D_0 = \int_{\mathbb{R}} |f(x)| e^{ \frac{1}{2} x } dx.$$

The result follows, taking $(A,B,C,D) = \frac{1}{m} (A_0,B_0,C_0, D_0).$
\end{proof}

\begin{cor}\label{C:AppendixC}
Fix a Schwartz function $h$ and $\mu_0>0$ sufficiently small. Then there is a constant $C=C(h,\mu_0)$ such that for all $g$ sufficiently large 
$$\sum_{r_n \text{ real}}  |h(r_n)| \leq C \; g \logp \left(\frac{1}{\mathrm{sys}(X)}\right) . $$
\end{cor}

\begin{proof}
Recall that $E_X \leq 2g-3$ by \cite{otalrosas}, and $N_X \leq 3g - 3$ as long as $\mu_0$ is sufficiently small. The result follows since 
$$ \sum_{k = 0}^\infty p(k) \sup_{[2k ,2k+2]} |h| $$
is finite for any function $p$ of polynomial growth. 
\end{proof}

\bibliography{MM}{}
\bibliographystyle{amsalpha}
\end{document}